\documentclass[11pt,reqno]{amsart}

\usepackage{amsmath,amssymb,mathrsfs}
\usepackage{graphicx,cite,xcolor}
\newtheorem{theorem}{Theorem}[section]
\newtheorem{corollary}[theorem]{Corollary}
\newtheorem{lemma}[theorem]{Lemma}

\newtheorem{remark}[theorem]{Remark}
\newtheorem{definition}[theorem]{Definition}

\numberwithin{equation}{section}

\begin{document}

\title [elastic scattering by multiple particles]{A Fast Solver for the Elastic
Scattering of Multiple Particles}

\author{Jun Lai}
\address{School of Mathematical Sciences, Zhejiang University, Hangzhou,
Zhejiang 310027, China}
\email{laijun6@zju.edu.cn}

\author{Peijun Li}
\address{Department of Mathematics, Purdue University, West Lafayette, IN 47907,
USA.}
\email{lipeijun@math.purdue.edu}

\subjclass[2010]{35P25, 45A05, 74J20, 74B05}

\keywords{Elastic wave equation, elastic obstacle scattering, boundary integral
equation, fast multiple method, Helmholtz decomposition}

\begin{abstract}
Consider the elastic scattering of a time-harmonic wave by multiple well
separated rigid particles in two dimensions. To avoid using the complex Green's
tensor of the elastic wave equation, we utilize the Helmholtz decomposition to
convert the boundary value problem of the elastic wave equation into a
coupled boundary value problem of Helmholtz equations. Based on single, double,
and combined layer potentials with the simpler Green's function of the Helmholtz
equation, we present three different boundary integral equations for the coupled
boundary value problem. The well-posedness of the new integral equations are
established. Computationally, a scattering matrix based method is proposed to
evaluate the elastic wave for arbitrarily shaped particles. The method uses the
local expansion for the incident wave and the multipole expansion for the
scattered wave. The linear system of algebraic equations is solved by GMRES
with fast multipole method (FMM) acceleration. Numerical results show that
the method is fast and highly accurate for solving the elastic scattering
problem with multiple particles.  
\end{abstract}

\maketitle

\section{Introduction}

A basic problem in scattering theory is the scattering of a time-harmonic wave
by an impenetrable medium, which is referred to as the obstacle scattering
problem \cite{CK-83}. It has played a fundamental role in many scientific areas
including radar and sonar (e.g., submarine detection), nondestructive testing
(e.g., detection of fatigue cracks in aircraft wings), remote sensing (e.g.,
monitoring deforestation), medical imaging (brain tumor detection), and
geophysical exploration (e.g., oil detection). Driven by these significant
applications, the obstacle scattering problems have been widely studied by
numerous researchers for all the three commonly used wave models: the Helmholtz
equation (acoustic waves), the Maxwell equation (electromagnetic waves), and
the Navier equation (elastic waves). Consequently, a great deal of mathematical
and numerical results are available \cite{N-01}. Recently, the scattering
problems for elastic waves have received ever increasing attention in both
engineering and mathematical communities for their important applications in
geophysics and seismology \cite{AH-SIAP76, ABG-15, HKS-IP13, LL-86, L-SIAP12,
PV-JASA, TC-JCP07}. The propagation of elastic waves is governed by the Navier
equation which is complex because of the coexistence of compressional and
shear waves with different wavenumbers. 

In many applications it is desirable to develop a computational model to simulate
the wave propagation in a medium consisting of multiple
particles \cite{HLZ-JCP13, HL-MMS10, PS-PRD73, M-06, JZ-CMAME12}, including the
application of imaging a target in a cluttered environment \cite{BHLZ-CM14} and
the design of composite materials with a specific wave response\cite{GG-JCP13}.
In this paper, we consider the two-dimensional elastic scattering problem of a
time-harmonic wave by multiple rigid obstacles which are embedded in a
homogeneous and isotropic elastic medium. The obstacles
are assumed to be well separated in the sense that each obstacle can be
circumscribed by a circle and all the circles are disjoint. The method
of boundary integral equations is employed to solve the elastic obstacle
scattering problem. Compared to finite difference or finite element
methods\cite{BXY-JCP17}, the boundary
integral method enjoys several intrinsic advantages: the solution is
characterized solely in terms of surface distributions so that there are fewer
unknowns; the radiation condition is implicitly and exactly imposed so as to
avoid the error that is introduced by using artificial radiation conditions
\cite{GK-WM90, GK-JCP04}. However, the Green's function of the elastic wave
equation is a second order tensor and is complicated to compute in the boundary
integral equations \cite{BLR-JCP14, YHX-SINUM16, TC-JCP09}. To avoid this issue,
we introduce two scalar potential functions and use the Helmholtz
decomposition to split the displacement of the wave field into the compressional
wave and the shear wave which satisfies the Helmholtz equation,
respectively \cite{YLLY-CiCP18}. Therefore the boundary value problem of the
Navier equation is converted equivalently into a coupled boundary value problem
of the Helmholtz equations for the potentials. Since the Green's function of the
Helmholtz equation is much simpler than that of the Navier equation,
it is computationally much easier to solve the Helmholtz system than to solve
the vectorial Navier equation. This simplification from the elastic Green's
function to the Helmholtz Green's function, however, does not come without cost.
Since the principal part of resulted boundary integral system is degenerated,
the Fredholm alternative can not be applied directly to obtain the existence
result of the system. By analyzing the properties of integral operators
thoroughly and introducing appropriate regularizers, we prove the well-posedness
for three different boundary integral formulations which are based on using the
single, double, and combined layer potentials. The theoretical analysis lays a
foundation on the numerical implementation of solving the elastic wave equation
based on the Helmholtz decomposition.

In numerical practice, the advantages of boundary integral methods can be offset
by the high computational cost incurred in evaluating the mutual interactions
among all elements. Moreover, each interaction involves singular
integrals whose analytical and/or numerical evaluation is expensive.  
In this work, we propose a fast and highly accurate numerical method for solving
the elastic scattering problem with multiple particles. The method extends the
classic multiple scattering theory for acoustic and electromagnetic waves to
elastic waves. It can handle many particles that are arbitrarily shaped and
randomly located in a homogeneous medium. The idea goes back to \cite{GG-JCP13,
LKG-OE14, LKB-JCP15} for the electromagnetic scattering of multiple particles.
For a given particle, we first use the integral formulation, which is based on
the Helmholtz decomposition, to construct a scattering matrix, which
is a matrix that maps the incoming wave to the outgoing wave. An important
feature of the matrix is that it only depends on the physical property of the
particle and is independent of the location and rotation of the particle, which 
suggests if all the particles are identical, up to a shift and rotation, the
scattering matrix only has to be computed once. With this matrix precomputed, we
then treat the outgoing scattering coefficients, instead of the discretization
points on the boundary of particles, as the unknowns in our equation. When
particles are in sub-wavelength regime and are well separated, this is highly
accurate with only about 20 unknowns per particle. Therefore it greatly reduces
the number of unknowns especially for particles with complicated geometry.
Moreover, the resulted system based on outgoing coefficients can be
preconditioned by the scattering matrix and the GMRES iterative solver becomes
extremely efficient after the preconditioning. The algorithm is further
accelerated by the fast multipole method FMM \cite{R-90}. Numerical experiments
show that for a given order of accuracy, the number of iterations grows linearly
with respect to the angular frequency for a fixed number of particles, and
increases sublinearly with respect to the number of particles for a fixed
angular frequency. Hence, the method is well suited for the elastic scattering
problem with multiple particles.

The paper is organized as follows. In Section 2, we introduce the
model equation for the elastic scattering by multiple obstacles. In particular,
the Helmholtz decomposition is utilized to convert the elastic wave equation
into a coupled Helmholtz system. Section 3 gives some preliminaries for boundary
integral operators. Section 4 is devoted to three different boundary integral
formulations for the coupled Helmholtz system. Their well-posedness are proved
based on the regularization theory and Fredholm alternative. In Section 5, a
scattering matrix based numerical method is proposed for solving the coupled
integral equation. Numerical experiments are presented in Section 6 to show the
performance of the proposed method. The paper is concluded with some general
remarks and a direction for future work in Section 7.

\section{Problem formulation}

Let us first specify the problem geometry which is shown in Figure
\ref{pg}. Consider the scattering problem for some two-dimensional elastically
rigid obstacles, the union of which is represented by a bounded domain $D$ with boundary
$\Gamma$. The infinite exterior domain $\mathbb R^2\setminus\overline D$ is
assumed to be filled with a homogeneous and isotropic elastic medium. In
particular, we assume that the domain $D$ consists of $M$ inclusions $D_j, j=1,
\dots, M$ which are bounded with smooth boundaries $\Gamma_j$, i.e.,
$D=\cup_{j=1}^M D_j$ and $\Gamma_j=\cup_{j=1}^M \Gamma_j$. Moreover, the
obstacles are assumed to be
well-separated, i.e., there exist balls $B_j$ such that $\overline D_j\subset
B_j, j=1, \dots, M$ and $B_i\cap B_j=\emptyset$ for $i\neq j$. Denote by
$\nu=(\nu_1, \nu_2)$ and $\tau=(\tau_1, \tau_2)$ the unit normal and tangential
vectors on $\Gamma$, respectively, where $\tau_1=-\nu_2$ and $\tau_2=\nu_1$.

\begin{figure}
\centering
\includegraphics[width=0.31\linewidth]{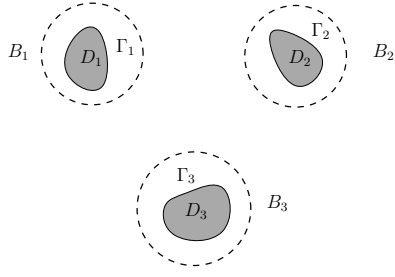}
\caption{Problem geometry of the elastic scattering by multiple obstacles.} 
\label{pg}
\end{figure}

Let the obstacles be illuminated by a time-harmonic plane wave ${\bf u}^{\rm
inc}$, which satisfies the two-dimensional Navier equation
\[
 \mu\Delta{\bf u}^{\rm inc}+(\lambda+\mu)\nabla\nabla\cdot{\bf
u}^{\rm inc}+\omega^2{\bf u}^{\rm inc}=0\quad\text{in}~\mathbb
R^2\setminus\overline D,
\]
where $\omega>0$ is the angular frequency and $\mu, \lambda$ are the
Lam\'{e} constants satisfying $\mu>0, \lambda+\mu>0$. It can be verified that
the incident wave ${\bf u}^{\rm inc}$ has the explicit expression
\[
 {\bf u}^{\rm inc}(x)=d e^{{\rm i}k_p x\cdot d}\quad\text{or}\quad
{\bf u}^{\rm inc}(x)=d^\perp e^{{\rm i}k_s x\cdot d},
\]
where the former is called the compressional plane wave and the latter is
referred to as the shear plane wave. Here $d=(\cos\theta, \sin\theta)$ is the
unit propagation direction vector, $\theta\in [0, 2\pi)$ is the incident angle,
$d^\perp=(-\sin\theta, \cos\theta)$ is an orthonormal vector of $d$, and 
\[
k_p=\omega/(\lambda+2\mu)^{1/2},\quad k_s=\omega/\mu^{1/2}
\]
are the compressional wavenumber and the shear wavenumber, respectively. More
generally, the incident field can be a linear combination of the compressional
and shear plane waves. 

The displacement of the total wave field ${\bf u}$ also satisfies the
Navier equation
\[
 \mu\Delta{\bf u}+(\lambda+\mu)\nabla\nabla\cdot{\bf
u}+\omega^2{\bf u}=0\quad\text{in}~\mathbb R^2\setminus\overline D.
\]
By assuming that each obstacle is impenetrable and rigid, we have 
\[
 {\bf u}=0\quad\text{on}~\Gamma. 
\]
The total field ${\bf u}$ consists of the incident field ${\bf
u}^{\rm inc}$ and the scattered field ${\bf v}$:
\[
 {\bf u}={\bf u}^{\rm inc}+{\bf v}. 
\]
It is easy to verify that the scattered field ${\bf v}$ satisfies the
Navier equation
\begin{equation}\label{nev}
\mu\Delta{\bf v}+(\lambda+\mu)\nabla\nabla\cdot{\bf
v}+\omega^2{\bf v}=0\quad\text{in}~\mathbb R^2\setminus\overline D
\end{equation}
and the boundary condition
\begin{equation}\label{bcv}
 {\bf v}=-{\bf u}^{\rm inc}\quad\text{on}~\Gamma. 
\end{equation}

Given a vector function ${\bf w}=(w_1, w_2)$ and a scalar function $w$,
define the scalar and vector curl operators
\[
 {\rm curl}{\bf w}=\partial_{x_1} w_2 - \partial_{x_1} w_1, \quad {\bf
curl}w=(\partial_{x_2}w, -\partial_{x_1}w).
\]
Let ${\bf v}$ be a solution of \eqref{nev}. Denote 
\[
{\bf v}_p=-\frac{1}{k_p^2}\nabla\nabla\cdot{\bf v},\quad
{\bf v}_s=\frac{1}{k_s^2}{\bf curl}{\rm curl}{\bf v},
\]
which are known as the compressional component and the shear component of
${\bf v}$, respectively. Since the scattering problem is imposed in the
open domain $\mathbb R^2\setminus\overline D$, the scattered field ${\bf
v}$ is required to satisfy the Kupradz--Sommerfeld radiation condition
\cite{S-49}, i.e., the components ${\bf v}_p$ and ${\bf v}_s$ are required to
satisfy the Sommerfeld radiation condition: 
\[
 \partial_\rho{\bf v}_p -{\rm i}k_p{\bf v}_p=o(\rho^{-1/2}),\quad
\partial_\rho{\bf v}_s -{\rm i}k_s{\bf v}_s=o(\rho^{-1/2}),\quad \rho=|x|.
\]

For any solution ${\bf v}$ of the elastic wave equation \eqref{nev}, we
introduce the Helmholtz decomposition
\begin{equation}\label{hd}
 {\bf v}=\nabla\phi+{\bf curl}\psi,
\end{equation}
where the scalar functions $\phi$ and $\psi$ are called potentials. Substituting
\eqref{hd} into \eqref{nev} yields 
\[
 \nabla((\lambda+2\mu)\Delta\phi+\omega^2\phi)+{\bf
curl}(\mu\Delta\psi+\omega^2\psi)=0\quad\text{in}~\mathbb
R^2\setminus\overline D,
\]
which is fulfilled if $\phi, \psi$ satisfy the Helmholtz equation
\[
 \Delta\phi+k_p^2\phi=0,\quad \Delta\psi+k_s^2\psi=0\quad\text{in}~\mathbb
R^2\setminus\overline
D.
\]
In addition, the potentials $\phi, \psi$ are required to satisfy
the Sommerfeld radiation condition
\[
 \partial_\rho\phi-{\rm i}k_p\phi=o(\rho^{-1/2}), \quad  \partial_\rho\psi-{\rm
i}k_s\psi=o(\rho^{-1/2}).
\]
Combining \eqref{bcv} and \eqref{hd} yields the boundary condition
\[
 {\bf v}=\nabla\phi+{\bf curl}\psi=-{\bf u}^{\rm
inc}\quad\text{on}~\Gamma.
\]
Taking the dot product of the above equation with $\nu$ and $\tau$,
respectively, and noting $\tau_1=-\nu_2, \tau_2=\nu_1$, we obtain a coupled
boundary condition for $\phi_1, \phi_2$ on $\Gamma$:
\[
 \partial_\nu\phi+\partial_\tau \psi=f,\quad
\partial_\tau\phi - \partial_\nu\psi=g,
\]
where 
\[
 f=-\nu\cdot{\bf u}^{\rm inc},\quad g=-\tau\cdot{\bf u}^{\rm
inc}. 
\]

Hence the obstacle scattering problem for elastic waves can be reduced
equivalently to the coupled boundary value problem of the Helmholtz equations:
\begin{equation}\label{bvp}
\begin{cases}
 \Delta\phi+k_p^2\phi=0,\quad \Delta\psi+k_s^2\psi=0 &\quad\text{in}~\mathbb
R^2\setminus\overline D,\\
\partial_\nu\phi+\partial_\tau \psi=f,\quad
\partial_\tau\phi - \partial_\nu\psi=g &\quad\text{on}~\Gamma,\\
\partial_\rho\phi-{\rm i}k_p\phi=o(\rho^{-1/2}), \quad \partial_\rho\psi-{\rm
i}k_s\psi=o(\rho^{-1/2})
&\quad\text{as}~\rho\to\infty.
\end{cases}
\end{equation}
The proof can be found in \cite{LWWZ-IP16} for the well-posedness of the above
scattering problem \eqref{bvp} by using the variational approach. In this work,
our goal is to develop a new and well-posed boundary integral equation, and
propose a fast numerical method to the scattering problem \eqref{bvp}. Hence we
assume that the boundary value problem \eqref{bvp} has a unique solution. 

\begin{theorem}\label{uniqueresult}
 The coupled Helmholtz system \eqref{bvp} has at most one solution for $k_s>0$ and $k_p> 0$. 
\end{theorem}

\section{Preliminaries of integral operators}

Let $\Gamma\subset\mathbb R^2$ be a smooth closed curve. Consider the
integral operators of the form \begin{eqnarray}\label{intequ1}
F(x) = \int_{\Gamma}K(x,x-y)\phi(y)ds(y)
\end{eqnarray}
and its adjoint with respect to $L^2(\Gamma)$
\begin{eqnarray}\label{intequ2}
G(x) = \int_{\Gamma}K(y,y-x)\phi(y)ds(y),
\end{eqnarray} 
where $K$ is an integral kernel and $\phi$ is called the density. The
following theorem can be found in \cite{N-01}.

\begin{theorem}\label{thm1}
Let $\alpha=(\alpha_1, \alpha_2)$ be a multi-index and $\beta$ be a positive
integer. Assume that the kernel $K$ in \eqref{intequ1}--\eqref{intequ2} is given
by 
\[
K(x,y)=h(x) y^{\alpha}|y|^{2\beta}\ln|y|,
\]
where $h(x)$ is a smooth function
defined on $\Gamma$. Then the kernel is of class $m=-(|\alpha|+2\beta+1)$. The
integral operator in \eqref{intequ1}--\eqref{intequ2} associated with the
kernel $K$ is continuous from $H^r(\Gamma)$ into $H^{r+m}(\Gamma)$ for any real
$r$.
\end{theorem}

Consider the two-dimensional Helmholtz equation
\begin{equation}\label{he}
\Delta u + k^2 u = 0\quad\text{in}~\mathbb R^2,
\end{equation}
where $k$ is the wavenumber satisfying $\Re(k)>0, \Im(k)\geq 0$.
It is known that the Green's function of \eqref{he} is 
\[ 
\Phi_k(x,y) = \frac{\rm i}{4}H^{(1)}_0(k|x-y|),
\]
where $H^{(1)}_0$ is the Hankel function of the first kind with order zero. 

Given a bounded domain $D\subset\mathbb R^2$ with smooth boundary $\Gamma$, let
$\nu$ and $\tau$ be the exterior unit normal vector and the unit tangential
vector of $\Gamma$, respectively. For $x\notin\Gamma$, define the single and
double layer potentials
\begin{align*}
\mathcal{S}_k\phi(x) &= \int_{\Gamma} \Phi_k(x,y)\phi(y)ds(y), \\
\mathcal{D}_k\phi(x) &= \int_{\Gamma} \frac{\partial \Phi_k(x,y)}{\partial
\nu(y)}\phi(y)ds(y),
\end{align*}
and the tangential boundary layer potential
\[
\mathcal{H}_k\phi(x) = \int_{\Gamma} \frac{\partial \Phi_k(x,y)}{\partial
\tau(y)}\phi(y)ds(y).
\]
For $k=0$, these potentials denote the layer potentials corresponding
to the two-dimensional Laplace equation where the Green's function is 
\[
 \Phi_0(x, y)=-\frac{1}{2\pi}\ln|x-y|. 
\]
Moreover, these potentials satisfy the well-known jump relations
\cite{CK-83}:
\begin{subequations}\label{jr}
\begin{align}
\lim_{x\rightarrow \Gamma^{\pm}}\mathcal{S}_k\phi(x) &= S_k\phi(x) =
\int_{\Gamma} \Phi_k(x,y)\phi(y)ds(y), \\
\lim_{x\rightarrow \Gamma^{\pm}}\mathcal{D}_k\phi(x) &= (\pm \frac{1}{2}+
D_k)\phi(x) = \pm \frac{1}{2}\phi(x)+\int_{\Gamma} \frac{\partial
\Phi_k(x,y)}{\partial \nu(y)}\phi(y)ds(y), \\
\lim_{x\rightarrow \Gamma^{\pm}}\frac{\partial
\mathcal{S}_k\phi(x)}{\partial \nu(x)} &=
(\mp \frac{1}{2}+ D'_k)\phi(x) = \mp \frac{1}{2}\phi(x)+\int_{\Gamma}
\frac{\partial \Phi_k(x,y)}{\partial \nu(x)}\phi(y)ds(y), \\
\lim_{x\rightarrow \Gamma^{\pm}} \mathcal{H}_k\phi(x) &=
H_k\phi(x)=\int_{\Gamma} \frac{\partial \Phi_k(x,y)}{\partial
\tau(y)}\phi(y)ds(y),
\end{align} 
\end{subequations}
where the plus sign means that $x$ approaches $\Gamma$ from the exterior and
the minus sign stands for that $x$ approaches $\Gamma$ from the interior. The
boundary operators $D_k$, $D'_k$, and $H_k$ are defined in the
sense of Cauchy principal value. In $L^2(\Gamma)$, $S_k$ is self-adjoint, i.e.,
$S_k=S_k'$, and $D'_k$ is the adjoint of $D_k$. The adjoint of $H_k$ is given
by 
\[
H'_k\phi(x) = \frac{\partial S_k\phi(x)} {\partial \tau(x)}  = \int_{\Gamma}
\frac{\partial \Phi_k(x,y)}{\partial \tau(x)}\phi(y)ds(y).
\]

For further investigation, it is indispensable to study the regularity of all
these boundary operators. We begin with the asymptotic form of the
Green function $\Phi_k(x,y)$ which can be found in \cite{OLBC-10}. 

\begin{lemma}\label{asyform}
When $k>0$, the Green's function $\Phi_k$ has the expansion
\[
\Phi_k(x,y)  =  \Phi_0(x,y) -\frac{k^2|x-y|^2}{4}\Phi_0(x,y) +
|x-y|^4p_1(|x-y|^2)\Phi_0(x,y) + p_2(|x-y|^2),
\]
where $p_1(x)$ and $p_2(x)$ are analytic functions.  
\end{lemma} 

Combining Theorem \ref{thm1} and Lemma \ref{asyform}, we obtain several useful
properties for the integral operators. The following results are related to the
regularity of boundary operators $D_k$, $H_k$ and their adjoint.

\begin{corollary}\label{coro1}
The following operators are bounded:
\begin{align*}
& D_0, D'_0: H^{r}(\Gamma) \rightarrow H^{r+s}(\Gamma),\\
& H_0, H'_0: H^{r}(\Gamma) \rightarrow H^{r}(\Gamma),\\
& D_k, D'_k: H^{r}(\Gamma) \rightarrow H^{r+3}(\Gamma),\\
& H_k, H'_k: H^{r}(\Gamma) \rightarrow H^{r}(\Gamma),
\end{align*}
where $r$ is an arbitrary real number and $s$ is an arbitrary positive real
number. 
\end{corollary}

\begin{proof}
We only show the proof of for the integral operators $D_k$ and $D_k'$, since the
results are standard and can be found in \cite{CK-83} for other integral
operators. It follows from Lemma \ref{asyform} that the kernel $D_k$
satisfies
\begin{align*}
\frac{\partial \Phi_k(x,y)}{\partial \nu(y)} &= \frac{\partial
\Phi_0(x,y)}{\partial \nu(y)} - \frac{|k|^2}{4}\frac{\partial
(|x-y|^2\Phi_0(x,y))}{\partial \nu(y)} + O\bigg(\frac{\partial
(|x-y|^4\Phi_0(x,y))}{\partial \nu(y)}\bigg) \\
&= K_0(x,y) - \frac{|k|^2}{4} K_1(x,y)+K_2(x,y),
\end{align*}
where $K_0$ is the kernel for the integral operator $D_0$ and is of class
$-\infty$, and $K_2$ is of class at most $-4$ by Theorem \ref{thm1}. A simple
calculation yields 
\begin{align*}
K_1(x,y) &= \frac{\partial (|x-y|^2\Phi_0(x,y))}{\partial \nu(y)}
\notag\\ &= \frac{\partial |x-y|^2}{\partial \nu(y)}\Phi_0(x,y) +
\frac{\partial \Phi_0(x,y)}{\partial \nu(y)}|x-y|^2 \\
&= -2(\nu(y)\cdot(x-y))\Phi_0(x,y)+ K_0(x,y)|x-y|^2.
\end{align*}
For a smooth curve $\Gamma$, it is shown in \cite{CK-83} that 
\[
\nu(y)\cdot(x-y) = O(|x-y|^2).
\]
Hence it follows from Theorem \ref{thm1} that $K_1$ is a kernel of class $-3$
and $D_k$ is a bounded operator from $H^{r}(\Gamma)$ to $H^{r+3}(\Gamma)$.
Similarly we can show that the kernel of $D'_k$ is also of class $-3$, which
completes the proof. 
\end{proof}

The following results are related to the properties of difference of boundary
operators. The proof is similar to that for Corollary \ref{coro1}, so we omit
it. 

\begin{corollary}\label{coro2}
The following mappings are bounded:
\begin{align*}
&D_k-D_0, D'_k-D'_0 : H^{r}(\Gamma) \rightarrow H^{r+3}(\Gamma),\\
&H_k-H_0, H'_k-H_0 : H^{r}(\Gamma) \rightarrow H^{r+2}(\Gamma),
\end{align*}
where $r$ is an arbitrary real number.
\end{corollary}

The next lemma follows from the property of Cauchy integrals \cite{K-99}. 

\begin{lemma}\label{cauchyprop}
Let $\Gamma\subset \mathbb{R}^2$ be a smooth curve. Then 
\begin{align*}
D_0^2-H_0^2 = \frac{I}{4}, \quad H_0 D_0 = -D_0 H_0, \quad 
{D'_0}^2-{H'_0}^2 = \frac{I}{4}, \quad {H'_0}{D'_0} = -D'_0 H'_0,
\end{align*}
where $I$ is the identity operator. 
\end{lemma}

In this paper, we mainly focus on functions in $H^{1/2}(\Gamma)$ and
$H^{-1/2}(\Gamma)$, which are the trace space of $H^{1}(D)$ and $L^2(D)$,
respectively. We also denote the vector function space with each component in
$H^{1/2}(\Gamma)$ by $H^{1/2}(\Gamma)^2$. Similar notation applies to
$H^{s}(\Gamma)$ for any real $s$. It is well known that the two dimensional
single layer boundary operator $S_0$, which is bounded from 
$H^{s}(\Gamma)$ to $H^{s+1}(\Gamma)$, is not invertible in general. However, we
have the following result which can be found in \cite{K-99}:

\begin{lemma}\label{s0posi}
There exists a constant $c>0$, which only depends on the curve $\Gamma$,
such that the operator $\overline{S}_0$, defined by
\[
(\overline{S}_0\phi)(x)=\int_{\Gamma}(\Phi_0(x,
y)+c)\phi(y)ds(y)=S_0\phi(x)+\int_{\Gamma} c \phi(y)ds(y),
\]
is invertible from $H^{s}(\Gamma)$ to $H^{s+1}(\Gamma)$ for any real $s$.
\end{lemma}

To this end, we denote the operator
\begin{eqnarray*}
\overline{\mathbf{S}}_0\overline{\mathbf{S}}_0=\begin{bmatrix}
 \overline{S}_0\overline{S}_0 & 0 \\
0 & \overline{S}_0\overline{S}_0
\end{bmatrix}
\end{eqnarray*} 
for a vector function ${\bf w}=(w_1, w_2)\in H^{s}(\Gamma)^2$ by
$\overline{\mathbf{S}}_0\overline{\mathbf{S}}_0$.  By Lemma \ref{s0posi}, the
operator $\overline{\mathbf{S}}_0\overline{\mathbf{S}}_0$ is invertible from
$H^{s}(\Gamma)^2$ to $H^{s+2}(\Gamma)^2$.

\section{Boundary integral equations}

In this section, we derive boundary integral equations for the scattering
problem \eqref{bvp} and show the well-posedness of the proposed boundary
integral equations. For clarity, we restrict our discussion to the scattering of
a single particle, which is still denoted by $D$ with boundary $\Gamma$.

Define two single layer potentials corresponding to the compressional and shear
wavenumbers:
\[
\phi(x) = \mathcal{S}_{k_p}\alpha(x), \quad \psi(x) = \mathcal{S}_{k_s}\beta(x),
\]
where $(\alpha(x), \beta(x))\in H^{-1/2}(\Gamma)^2$ are densities. Using the boundary condition
\eqref{bcv} and the jump relations \eqref{jr}, we obtain the integral equation 
\begin{eqnarray}\label{mainint}
A\begin{bmatrix}
\alpha(x) \\
\beta(x)
\end{bmatrix}=
\begin{bmatrix}
-\frac{I}{2}+D'_{k_p} & H'_{k_s} \\
H'_{k_p} & \frac{I}{2}-D'_{k_s}
\end{bmatrix}
\begin{bmatrix}
\alpha(x) \\
\beta(x)
\end{bmatrix}
=
\begin{bmatrix}
f(x) \\
g(x)
\end{bmatrix}.
\end{eqnarray}
where $I$ is the identity operator. We first state the following existence result for equation \eqref{mainint}.
\begin{theorem}\label{wp}
	Assume that neither $k_s$ or $k_p$ is the eigenvalue of the interior Dirichlet
	problem for the Helmholtz equation in $D$. Then the integral equation
	\eqref{mainint} has a unique solution in $H^{-1/2}(\Gamma)^2$.
\end{theorem}
\begin{remark}
	It is easy to see that $A=A_0+K$, where
	\begin{eqnarray}
	A_0=
	\begin{bmatrix}
	-\frac{I}{2} & H'_{0} \\
	H'_{0} & \frac{I}{2}
	\end{bmatrix}
	\end{eqnarray}
	is bounded in $H^{-1/2}(\Gamma)^2$ and $K$ is a compact operator in $H^{-1/2}(\Gamma)^2$. If $A_0$ is invertible, Fredholm alternative can be directly applied to show the invertibility of $A$. However,  $A_0$ is degenerated in the sense that
	\begin{eqnarray}
	A^2_0=\begin{bmatrix}
	{D'}^2_0 & 0 \\
	0 & {D'}^2_0
	\end{bmatrix},
	\end{eqnarray} 
	where $D'_0$ is a smooth operator by Corollary  \ref{coro1}.
\end{remark}
The existence result also holds for the integral
representation by using double layer potentials
\[
\phi(x) = \mathcal{D}_{k_p}\alpha(x), \quad \psi(x) = \mathcal{D}_{k_s}\beta(x).
\]
Using the jump relations \eqref{jr}, we obtain the integral equation
\begin{eqnarray}\label{mainint2}
M\begin{bmatrix}
\alpha(x) \\
\beta(x)
\end{bmatrix}
=
\begin{bmatrix}
f(x) \\
g(x)
\end{bmatrix},
\end{eqnarray}
where the coefficient matrix
\[
M=\begin{bmatrix}
T_{k_p} & \frac{1}{2}\partial_{\tau}+\partial_{\tau} D_{k_s} \\
\frac{1}{2}\partial_{\tau}+\partial_{\tau} D_{k_p} & -T_{k_s}
\end{bmatrix}.
\]
and $T_k= \partial_{\nu} D_{k}$.
It holds the following  existence result.
\begin{theorem}\label{mainthm2}
	If neither $k_s$ or $k_p$ is the eigenvalue of the interior
	Neumann problem for the Helmholtz equation in $D$, the integral equation
	\eqref{mainint2} has a unique solution in $H^{1/2}(\Gamma)^2$.
\end{theorem}
To remove the assumption of Theorem \ref{wp} or \ref{mainthm2},
we propose a combined double and single layer representation to obtain a
uniquely solvable integral system for any $k_s$ and $k_p$. Consider the
combined layer potentials
\[
\phi(x) = (\mathcal{D}_{k_p}-{\rm i}\mathcal{S}_{k_p})\alpha(x), \quad \psi(x) =
(\mathcal{D}_{k_s}-{\rm i}\mathcal{S}_{k_s})\beta(x),
\] 
which results in a combined integral equation
\begin{align}\label{combsys}
(M-{\rm i}A)\begin{bmatrix}
\alpha(x) \\
\beta(x)
\end{bmatrix}
&= \left(\begin{bmatrix}
T_{k_p} & \frac{1}{2}\partial_{\tau}+\partial_{\tau} D_{k_s} \\
\frac{1}{2}\partial_{\tau}+\partial_{\tau} D_{k_p} & -T_{k_s}
\end{bmatrix}
-{\rm i}\begin{bmatrix}
-\frac{1}{2}+D'_{k_p} & H'_{k_s} \\
H'_{k_p} & \frac{1}{2}-D'_{k_s}
\end{bmatrix}\right)
\begin{bmatrix}
\alpha(x) \\
\beta(x)
\end{bmatrix} \notag\\
&=
\begin{bmatrix}
f(x) \\
g(x)
\end{bmatrix}.
\end{align}
We have the following existence result.
\begin{theorem}\label{mainthm3}
For any $k_p>0$ and $k_s>0$, the integral equation \eqref{combsys} admits
a unique solution in $H^{1/2}(\Gamma)^2$.
\end{theorem}

In what follows, we discuss the proofs of Theorems \ref{wp}, \ref{mainthm2} and
\ref{mainthm3} in details.

\subsection{Proof of Theorem \ref{wp}}

To construct an appropriate regularizer for the operator $A$, we consider the
interior problem of the coupled Helmholtz system
\begin{eqnarray}\label{interiorequ}
\begin{cases}
\Delta \phi +k_p^2\phi=0, \quad \Delta \psi +k_s^2\psi=0 &\quad \text{in}~ D,\\
-\partial_\nu \phi + \partial_\tau \psi = f, \quad
\partial_\tau \phi + \partial_\nu \psi = g
&\quad \text{on} ~ \Gamma.  
\end{cases}
\end{eqnarray}
Assume the solutions $\phi$ and $\psi$ have the integral representations
\[
\phi = S_{k_p}\alpha(x), \quad \psi = S_{k_s}\beta(x).
\]
Using the boundary condition, we obtain the integral equation
\[B\begin{bmatrix}
\alpha \\
\beta
\end{bmatrix}=
\begin{bmatrix}
f \\
g
\end{bmatrix},
\mbox{ where }
B=\begin{bmatrix}
-\frac{I}{2}-D'_{k_p} & H'_{k_s} \\
H'_{k_p} & \frac{I}{2}+D'_{k_s}
\end{bmatrix}.
\]
To prove Theorem \ref{wp}, we also need to derive the adjoint operator of $A$ in
$L^2(\Gamma)^2$. By applying Green's identity to equation \eqref{interiorequ},
it holds for
$x\in\Gamma^-$ that 
\begin{align*}
\left(\frac{I}{2}+D_{k_p}\right)\phi - S_{k_p} \partial_\nu \phi & = 
\phi,\\
\left(\frac{I}{2}+D_{k_s}\right)\psi - S_{k_s} \partial_\nu \psi & = 
\psi.
\end{align*}
It follows from the boundary condition in equation \eqref{interiorequ} that we have
\begin{align*}
\left(-\frac{I}{2}+D_{k_p}\right)\phi - S_{k_p} \partial_\tau \psi & =
-S_{k_p}f, \\
\left(-\frac{I}{2}+D_{k_s}\right)\psi + S_{k_s} \partial_\tau \phi & =
S_{k_s}g.
\end{align*}
Noting $H_k\phi = -S_k\partial_\tau \phi$, we obtain
\[
A'\begin{bmatrix}
\phi \\
\psi
\end{bmatrix}=\begin{bmatrix}
-\frac{I}{2}+D_{k_p} & H_{k_p} \\
H_{k_s} & \frac{I}{2}-D_{k_s}
\end{bmatrix}
\begin{bmatrix}
\phi \\
\psi
\end{bmatrix}
=
\begin{bmatrix}
-S_{k_p}f \\
-S_{k_s}g
\end{bmatrix}.
\]
It is easy to check that the operator $A'$ is the adjoint of the operator $A$
with respect to the bilinear form in $L^2(\Gamma)^2$ given by 
\[
\langle{\bf u}, {\bf v} \rangle = \int_{\Gamma} \big(u_1 v_1
+u_2 v_2\big)ds.
\]
where ${\bf u}=(u_1, u_2)$ and ${\bf v}=(v_1, v_2)$.
\begin{theorem}\label{coercivity_thm}
For any vector function ${\bf f}\in
H^{-1/2}(\Gamma)^2$, the operators $A, B$ satisfy
\begin{align*}
 (AB){\bf f} =
\left(-\frac{(k^2_s+k^2_p)}{2}\overline{\mathbf{S}}_0\overline{\mathbf{S}}
_0+K_1\right){\bf f},\\
 (BA){\bf f} =
\left(-\frac{(k^2_s+k^2_p)}{2}\overline{\mathbf{S}}_0\overline{\mathbf{S}}
_0+K_2\right){\bf f},
\end{align*}
where $K_1, K_2$ are compact operators from $H^{-1/2}(\Gamma)^2$ to
$H^{3/2}(\Gamma)^2$.
\end{theorem}

\begin{proof} 
It follows from a straightforward calculation that
\begin{align*}
AB&=
\begin{bmatrix}
-\frac{I}{2}+D'_{k_p} & H'_{k_s} \\
H'_{k_p} & \frac{I}{2}-D'_{k_s}
\end{bmatrix}
\begin{bmatrix}
-\frac{I}{2}-D'_{k_p} & H'_{k_s} \\
H'_{k_p} & \frac{I}{2}+D'_{k_s}
\end{bmatrix}\\
&=\begin{bmatrix}
\frac{I}{4}-(D'_{k_p})^2+H'_{k_s}H'_{k_p} & D'_{k_p}H'_{k_s}+H'_{k_s}D'_{k_s} \\
-D'_{k_s}H'_{k_p}-H'_{k_p}D'_{k_p} & \frac{I}{4}-(D'_{k_s})^2+H'_{k_p}H'_{k_s}
\end{bmatrix}.
\end{align*} 	
We first look at the off diagonal elements. It can be verified that 
\begin{align*}
D'_{k_p}H'_{k_s}+H'_{k_s}D'_{k_s}  = &
(D'_{k_p}-D'_0)H'_0+D'_{k_p}(H'_{k_s}-H'_0) \\
& \quad + H'_{k_s}(D'_{k_s}-D'_0)+(H'_{k_s}-H'_0)D'_{0},
\end{align*}
where $D'_0H'_0+H'_0D'_0$ vanishes due to Lemma \ref{cauchyprop}. It follows
from Corollaries \ref{coro1} and \ref{coro2} that
$(D'_{k_p}-D'_0)H'_0, H'_{k_s}(D'_{k_s}-D'_0)$ are bounded operators from
$H^{-1/2}(\Gamma)$ to $H^{-1/2+3}(\Gamma)$ and
$D'_{k_p}(H'_{k_s}-H'_0), (H'_{k_s}-H'_0)D'_{0}$ are bounded operators 
from $H^{-1/2}(\Gamma)$ to $H^{-1/2+5}(\Gamma)$. Therefore,
$D'_{k_p}H'_{k_s}+H'_{k_s}D'_{k_s}$ is a compact
operator from $H^{-1/2}(\Gamma)$ to $H^{3/2}(\Gamma)$. Similarly, we can show that 
$-D'_{k_s}H'_{k_p}-H'_{k_p}D'_{k_p}$ is also a compact operator from
$H^{-1/2}(\Gamma)$ to $H^{3/2}(\Gamma)$.

Next we check the diagonal elements. Using Lemma \ref{cauchyprop}, we obtain
\begin{align*}
\frac{I}{4}-(D'_{k_p})^2+H'_{k_s}H'_{k_p} = &
(H'_{k_s}-H'_0)H'_{0}+H'_{k_s}(H'_{k_p}-H'_0)\\
&\quad -(D'_{k_p})^2+(D'_{0})^2,\\
\frac{I}{4}-(D'_{k_s})^2+H'_{k_p}H'_{k_s} = &
(H'_{k_p}-H'_0)H'_{0}+H'_{k_p}(H'_{k_s}-H'_0)\\
&\quad -(D'_{
k_s})^2+(D'_{0})^2.
\end{align*}
From Corollaries \ref{coro1} and \ref{coro2}, the operators $(D'_{k_p})^2$, $(D'_{0})^2$ are bounded 
from $H^{-1/2}(\Gamma)$ to $H^{-1/2+6}(\Gamma)$. Hence they are both compact
from $H^{-1/2}(\Gamma)$ to $H^{3/2}(\Gamma)$. Consider the operator 
\begin{equation}\label{operatora11}
(H'_{k_s}-H'_0)H'_{0}+ H'_{k_s}(H'_{k_p}-H'_0).
\end{equation}
Clearly, it is bounded from $H^{-1/2}(\Gamma)$ to $H^{3/2}(\Gamma)$. Using the
asymptotic form in Lemma \ref{asyform}, we have the following decomposition
\begin{align*}
(H'_{k_s}-H'_0)H'_{0}\phi(x) & =  -\frac{\partial }{\partial
\tau(x)}\int_{\Gamma} \frac{|k_s(x-y)|^2}{4}\Phi_0(x,y)\\
&\qquad \times\frac{\partial}{\partial
\tau(y)}\int_{\Gamma}\Phi_0(y,z)\phi(z)ds(z)ds(y)+K_1\phi(x),\\
H'_{k_s}(H'_{k_p}-H'_0)\phi(x) & =  -\frac{\partial }{\partial
\tau(x)}\int_{\Gamma} \Phi_0(x,y)\\
&\qquad \times\frac{\partial }{\partial
\tau(y)}\int_{\Gamma}\frac{|k_p(y-z)|^2}{4}\Phi_0(y,z)\phi(z)ds(z)ds(y)
+K_2\phi(x), 
\end{align*}
where $K_1$ and $K_2$ are compact operators from $H^{-1/2}(\Gamma)$ to
$H^{3/2}(\Gamma)$. For the first operator in the right hand side of
$(H'_{k_s}-H'_0)H'_{0}$, we note 
\begin{align*}
&-\frac{\partial }{\partial \tau(x)}\int_{\Gamma}
\frac{|k_s(x-y)|^2}{4}\Phi_0(x,y)\frac{\partial }{\partial
\tau(y)}\int_{\Gamma}\Phi_0(y,z)\phi(z)ds(z)ds(y)\\
&=  \int_{\Gamma} \frac{\partial }{\partial \tau(x)}\bigg(\frac{\partial
}{\partial \tau(y)}\frac{|k_s(x-y)|^2}{4}\Phi_0(x,y)\bigg)\int_{\Gamma}\Phi_0(y,
z)\phi(z)ds(z)ds(y) \\
&= \int_{\Gamma} \bigg(\frac{\partial }{\partial \tau(x)}+\frac{\partial
}{\partial \tau(y)}\bigg)\bigg(\frac{\partial }{\partial
\tau(y)}\frac{|k_s(x-y)|^2}{4}\Phi_0(x,y)\bigg)\int_{\Gamma}\Phi_0(y,
z)\phi(z)ds(z)ds(y) \\
&\qquad - \int_{\Gamma} \frac{\partial }{\partial \tau(y)}\bigg(\frac{\partial
}{\partial \tau(y)}\frac{|k_s(x-y)|^2}{4}\Phi_0(x,y)\bigg)\int_{\Gamma}\Phi_0(y,
z)\phi(z)ds(z)ds(y) \\
& = M\phi(x) +N\phi(x),
\end{align*}
where $M$ denotes the first operator and $N$ denotes the second one.

We show that $M$ is a compact operator from $H^{-1/2}(\Gamma)$ to
$H^{3/2}(\Gamma)$. In fact, it holds
\begin{align*}
& \bigg(\frac{\partial }{\partial \tau(x)}+\frac{\partial }{\partial
\tau(y)}\bigg)\bigg(\frac{\partial }{\partial
\tau(y)}\frac{|k_s(x-y)|^2}{4}\Phi_0(x,y)\bigg) \\
& =  -\frac{k_s^2}{4\pi}(1-\tau(x)\tau(y))\ln(|x-y|)+O((x-y)\ln(|x-y|)) \\
& =  O((x-y)\ln(|x-y|)).
\end{align*}
By Theorem \ref{thm1}, $M$ is bounded from  $H^{-1/2}(\Gamma)$ to
$H^{5/2}(\Gamma)$ which implies that $M$ is compact from $H^{-1/2}(\Gamma)$ to
$H^{3/2}(\Gamma)$. For the operator $N$, it is clear to note that 
\begin{align*}
&-\frac{\partial }{\partial \tau(y)}\bigg(\frac{\partial }{\partial
\tau(y)}\frac{|k_s(x-y)|^2}{4}\Phi_0(x,y)\bigg) \\
& =  -\frac{k_s^2}{2}\Phi_0(x-y)+O((x-y)\ln(|x-y|)).
\end{align*}
Therefore,
\begin{eqnarray}\label{operatorN}
N\phi(x) = -\frac{k_s^2}{2}S_0S_0\phi(x)+K\phi(x),
\end{eqnarray}
where $K$ is compact from $H^{-1/2}(\Gamma)$ to $H^{3/2}(\Gamma)$.
Similarly, the following property can be shown for the operator $H'_{k_s}(H'_{k_p}-H'_0)$
\begin{eqnarray*}
\left(H'_{k_s}(H'_{k_p}-H'_0)\right)\phi(x) = -\frac{k_p^2}{2}S_0S_0\phi(x)+K\phi(x)
\end{eqnarray*}

Combining \eqref{operatora11}--\eqref{operatorN}, we obtain
\begin{eqnarray*}
 \left(\frac{I}{4}-(D'_{k_p})^2+H'_{k_s}H'_{k_p}\right)\phi=\left(-\frac{
(k_s^2+k_p^2)}{2}S_0S_0+K\right)\phi,
\end{eqnarray*}
where $K$ is a compact operator from $H^{-1/2}(\Gamma)$ to $H^{3/2}(\Gamma)$. Following the same argument, we can show
\begin{eqnarray*}
\left(\frac{I}{4}-(D'_{k_s})^2+H'_{k_p}H'_{k_s}\right)\phi=\left(-\frac{
(k_s^2+k_p^2)}{2}S_0S_0+K\right)\phi,
\end{eqnarray*}
which proves the first part of the theorem since $\overline{S}_0$ 
and $S_0$ only differ by a smooth operator.

For the second part, we have from straightforward calculations that 
\begin{align*}
BA&=
\begin{bmatrix}
-\frac{I}{2}-D'_{k_p} & H'_{k_s} \\
H'_{k_p} & \frac{I}{2}+D'_{k_s}
\end{bmatrix}
\begin{bmatrix}
-\frac{I}{2}+D'_{k_p} & H'_{k_s} \\
H'_{k_p} & \frac{I}{2}-D'_{k_s}
\end{bmatrix}
\\
&=\begin{bmatrix}
\frac{I}{4}-(D'_{k_p})^2+H'_{k_p}H'_{k_s} & -D'_{k_p}H'_{k_s}-H'_{k_s}D'_{k_s}
\\
D'_{k_s}H'_{k_p}+H'_{k_p}D'_{k_p} & \frac{I}{4}-(D'_{k_s})^2+H'_{k_p}H'_{k_s}
\end{bmatrix}.
\end{align*} 	
The rest of the proof is the same as the first part and is omitted here. 
\end{proof}

Next we consider the adjoint operator $A'$ and introduce the
operator
\begin{eqnarray*}
B'=\begin{bmatrix}
-\frac{I}{2}-D_{k_p} & H_{k_p} \\
H_{k_s} & \frac{I}{2}+D_{k_s}
\end{bmatrix},
\end{eqnarray*}
which is the adjoint of operator $B$ in $L^2(\Gamma)^2$. Following exactly the
same argument, we have the following result. 

\begin{theorem}
For any vector function ${\bf f}\in H^{1/2}(\Gamma)^2$, the operators $A', B'$
satisfy
\begin{align*}
(A'B'){\bf
f}=\left(-\frac{(k^2_s+k^2_p)}{2}\overline{\mathbf{S}}_0\overline{\mathbf{S}}
_0+K_1\right){\bf f} \\
(B'A'){\bf
f}=\left(-\frac{(k^2_s+k^2_p)}{2}\overline{\mathbf{S}}_0\overline{\mathbf{S}}
_0+K_2\right){\bf f} 
\end{align*}
where $K_1, K_2$ are compact operators from $H^{1/2}(\Gamma)^2$ to
$H^{5/2}(\Gamma)^2$.
\end{theorem}

Since $\overline{\mathbf{S}}_0\overline{\mathbf{S}}_0$ is invertible from
$H^{s}(\Gamma)^2$ to $H^{s+2}(\Gamma)^2$ with $s\in \mathbb{R}$, by the 
Fredholm alternative, the operators $A$ and $A'$ have finite dimensional
null spaces and their ranges are given by 
\begin{align*}
{\rm Ran}(A) & = \{{\bf f}\in H^{-1/2}(\Gamma)^2:
\langle {\bf f}, {\bf g}\rangle=0,\, {\bf g}\in
{\rm Ker}(A') \}, \\ 
{\rm Ran}(A') & = \{{\bf f}\in H^{1/2}(\Gamma)^2:
\langle{\bf f}, {\bf h}\rangle=0,\, {\bf h}\in
{\rm Ker}(A) \}.
\end{align*}
The kernel of $A$ and $A'$ are given in the following theorem.

\begin{theorem}\label{nullthm}
If neither $k_s$ or $k_p$ is the eigenvalue of the interior Dirichlet
problem for the Helmholtz equation in $D$, then ${\rm Ker}(A)=
{\rm Ker}(A')=\{0\}$. 
\end{theorem}

\begin{proof}
Assume $(\alpha(x),\beta(x))\in H^{-1/2}(\Gamma)^2$ satisfies
\[
 A\begin{bmatrix}
   \alpha\\
   \beta
  \end{bmatrix}
=\begin{bmatrix}
   0\\
   0
  \end{bmatrix}.
\]
Let 
\[
\phi(x) = \mathcal{S}_{k_p}\alpha(x), \quad \psi(x) =
\mathcal{S}_{k_s}\beta(x), \quad x\in\mathbb R^2\setminus\Gamma.
\]
Then $(\phi, \psi)$ satisfies \eqref{bvp} with $f=0, g=0.$ By the uniqueness
result in Theorem \ref{uniqueresult}, it holds
\[
\phi(x)=\psi(x)=0, \quad x\in\mathbb R^2\setminus\overline D.
\]
It follows from the continuity of single layer potential that
$\phi(x)=\psi(x)=0$ for $x\in \Gamma^-$. Since neither $k_s$ or $k_p$ is the
eigenvalue of the interior Dirichlet problem in $D$, we have $\phi(x)=\psi(x)=0$
for $x\in D$. Using the jump relation of double layer potential, we obtain
$\alpha(x)=\beta(x)=0$, which implies ${\rm Ker}(A)=\{0\}$.

Now assume $(\alpha(x),\beta(x))\in{\rm Ker}(A')$. Let $x\in\mathbb
R^2\setminus\Gamma$ and consider 
\begin{align*}
\phi (x) & = \mathcal{D}_{k_p}\alpha(x)-\mathcal{S}_{k_p} \partial_\tau
\beta(x), \\
\psi (x) & = \mathcal{D}_{k_s}\beta(x)+\mathcal{S}_{k_s} \partial_\tau
\alpha(x).
\end{align*}
Since $\phi(x)=\psi(x)=0$ when $x$ approaches $\Gamma$ from the
interior, by assumption, it holds $\phi(x)=\psi(x)=0$ for $x\in D$. By
Green's theorem, when $x$ approaches $\Gamma$ from the exterior, i.e.,
$x\rightarrow \Gamma^+$, we have
\begin{subequations}\label{iden1}
\begin{align}
\phi(x) = \alpha(x), & \quad \psi(x) = \beta(x), \\
  \partial_\nu \phi(x)=\partial_\tau
\beta(x), & \quad \partial_\nu \psi(x)=-\partial_\tau
\alpha(x),
\end{align}
\end{subequations}
which shows that $\phi$ and $\psi$ satisfies \eqref{bvp} with $f=g=0$.
Therefore, by the uniqueness of the scattering problem, $\phi(x)=\psi(x)=0$ in
$\mathbb R^2\setminus\overline D$. Following \eqref{iden1} it yields 
that $\alpha(x)=\beta(x)=0$, which completes the proof.   
\end{proof}

The well-posedness of the integral equation \eqref{mainint} follows immediately from the
Fredholm alternative, which completes the proof of Theorem \ref{wp}.

\begin{remark}
 In practice, the assumption of Theorem \ref{wp} may be violated for a given
domain $D$. Besides using the combined layer representation as given in Theorem
\ref{mainthm3}, this issue can also be resolved based on a  modified single
layer representation  when domain $D$ is simply connected and $0\in D$.
Following the idea in \cite{CK-83}, we can modify the integral representation
to 
\begin{subequations}\label{nonull}
\begin{align}
\phi(x) = \mathcal{S}_{k_p}\alpha(x) +\sum_{n=0}^{\infty}a_n
H^{(1)}_n(k_p|x|)e^{{\rm i}n\frac{x}{|x|}}
\int_{\Gamma}H^{(1)}_n(k_p|y|)e^{{\rm i}n\frac{y}{|y|}}\alpha(y)ds(y),\\
\psi(x) = \mathcal{S}_{k_s}\beta(x)+ \sum_{n=0}^{\infty}b_n
H^{(1)}_n(k_s|x|)e^{{\rm i}n\frac{x}{|x|}}
\int_{\Gamma}H^{(1)}_n(k_s|y|)e^{{\rm i}n\frac{y}{|y|}}\beta(y)ds(y),
\end{align}
\end{subequations} 
where $\hat x=x/|x|, \hat y=y/|y|$ and $H_n^{(1)}$ is the Hankel function of
the first kind with order $n$. Under some appropriate assumptions on $\{a_n\}$
and $\{b_n\}$, it can be shown the the representation \eqref{nonull} is free of
resonance. Readers are referred to \cite{CK-83} for more details.
\end{remark}

\subsection{Proof of Theorem \ref{mainthm2}}

Introduce the operator
\begin{eqnarray*}
N=
\begin{bmatrix}
S_{k_p} & 0 \\
0 & S_{k_s}
\end{bmatrix}
\begin{bmatrix}
T_{k_p}& \frac{1}{2}\partial_{\tau}-\partial_{\tau} D_{k_s} \\
\frac{1}{2}\partial_{\tau}-\partial_{\tau} D_{k_p} & -T_{k_s}
\end{bmatrix}
\begin{bmatrix}
S_{k_p} & 0 \\
0 & S_{k_s}
\end{bmatrix}.
\end{eqnarray*}
We show that $N$ is a regularizer of $M$ in \eqref{mainint2}. Let us
begin with the Calderon identities which can be found in \cite{N-01}.

\begin{lemma}\label{caldlemma}
Let $I$ be the identity operator and $k$ be the wavenumber with $k>0$. For a smooth closed curve $\Gamma\subset \mathbb{R}^2$, it holds
that 
\[
D^2_k-S_{k}T_k=\frac{I}{4}, \quad D'^2_k-T_{k}S_k=\frac{I}{4}.
\]
\end{lemma}

Analogous to Theorem \ref{coercivity_thm}, we have the following result for 
operators $M$ and $N$.

\begin{theorem}\label{doublethm}
For any vector function ${\bf f}\in H^{1/2}(\Gamma)^2$, the operators $M$ and $N$
satisfy  
\begin{align*}
(NM){\bf
f}=\left(-\frac{(k_s^2+k_p^2)}{8}\overline{\mathbf{S}}_0\overline{\mathbf{S}}
_0+K_1\right){\bf f}, \\
(MN){\bf
f}=\left(-\frac{(k_s^2+k_p^2)}{8}\overline{\mathbf{S}}_0\overline{\mathbf{S}}
_0+K_2\right){\bf f},
\end{align*}
where $K_1, K_2$ are compact operators from $H^{1/2}(\Gamma)^2$ to
$H^{5/2}(\Gamma)^2$.
\end{theorem}

\begin{proof}
It follows from a tedious but straightforward calculation that  
\begin{align*}
NM &= \begin{bmatrix}
S_{k_p}T_{k_p}S_{k_p} &
\frac{1}{2}S_{k_p}\partial_{\tau}S_{k_s}-S_{k_p}\partial_{\tau}D_{k_s}S_{k_s}
\\
\frac{1}{2}S_{k_s}\partial_{\tau}S_{k_p}-S_{k_s}\partial_{\tau}D_{k_s}
S_{k_p} & -S_{k_s}T_{k_s}S_{k_s}
\end{bmatrix}
\begin{bmatrix}
T_{k_p} & \frac{1}{2}\partial_{\tau}+\partial_{\tau} D_{k_s} \\
\frac{1}{2}\partial_{\tau}+\partial_{\tau} D_{k_p} & -T_{k_s}
\end{bmatrix}
\notag \\
&=\begin{bmatrix}
A_{11} & A_{12} \\
A_{21} & A_{22}
\end{bmatrix},
\end{align*}
where 
\begin{align*}
A_{11} &=
S_{k_p}T_{k_p}S_{k_p}T_{k_p}+\frac{1}{4}S_{k_p}\partial_{\tau}S_{k_s}\partial_{
\tau}-S_{k_p}\partial_{\tau}D_{k_s}S_{k_s}\partial_{\tau}D_{k_p}+\frac{1}{2}S_{
k_p}\partial_{\tau}S_{k_s}\partial_{\tau}D_{k_p}-\frac{1}{2}S_{k_p}\partial_{
\tau}D_{k_s}S_{k_s}\partial_{\tau},\\
A_{12}
&= \frac{1}{2}S_{k_p}T_{k_p}S_{k_p}\partial_{\tau}-\frac{1}{2}S_{k_p}\partial_{
\tau }S_{k_s}T_{k_s}+S_{k_p}T_{k_p}S_{k_p}\partial_{\tau}
D_{k_s}+S_{k_p}\partial_{\tau}D_{k_s}S_{k_s}T_{k_s}, \\
A_{21} &=  \frac{1}{2}S_{k_s}\partial_{\tau}S_{k_p}T_{k_p}-
\frac{1}{2}S_{k_s}T_{k_s}S_{k_s}\partial_{\tau}-S_{k_s}\partial_{\tau}D_{k_s}S_{
k_p}T_{k_p} -S_{k_s}T_{k_s}S_{k_s}\partial_{\tau}D_{k_p}, \\
A_{22} &=
\frac{1}{4}S_{k_s}\partial_{\tau}S_{k_p}\partial_{\tau}-\frac{1}{2}S_{k_s}
\partial_{\tau}D_{k_s}S_{k_p}\partial_{\tau} +
\frac{1}{2}S_{k_s}\partial_{\tau}S_{k_p}\partial_{\tau}D_{k_s}-S_{k_s}\partial_{
\tau}D_{k_s}S_{k_p}\partial_{\tau}D_{k_s} + S_{k_s}T_{k_s}S_{k_s}T_{k_s}.
\end{align*}

For $A_{11}$, noting $S_{k}\partial\tau= -H_{k}$, $\partial{\tau}S_{k}= H'_{k}$,
and using Corollaries \ref{coro1} and \ref{coro2}, we can show that the
following operators are bounded:
\begin{align*}
S_{k_p}\partial_{\tau}D_{k_s}S_{k_s}\partial_{\tau}D_{k_p}=H_{k_p}D_{k_s}H_{
k_s}D_{k_p} : H^{1/2}(\Gamma)\rightarrow H^{1/2+6}(\Gamma), \\
S_{k_p}\partial_{\tau}S_{k_s}\partial_{\tau}D_{k_p} = H_{k_p}H_{k_s}D_{k_p} :
H^{1/2}(\Gamma)\rightarrow H^{1/2+3}(\Gamma), \\
S_{k_p}\partial_{\tau}D_{k_s}S_{k_s}\partial_{\tau}  = 
H_{k_p}D_{k_s}H_{k_s} : H^{1/2}(\Gamma)\rightarrow H^{1/2+3}(\Gamma),
\end{align*} 
which implies that they are all compact from $H^{1/2}(\Gamma)$ to
$H^{5/2}(\Gamma)$. For the first two operators in $A_{11}$, combining Lemmas
\ref{cauchyprop} and \ref{caldlemma} leads to 
\begin{align*}
S_{k_p}T_{k_p}S_{k_p}T_{k_p}+\frac{1}{4}S_{k_p}\partial_{\tau}S_{k_s}\partial_{
\tau}&= \left(D^2_{k_p}-\frac{I}{4}\right)\left(D^2_{k_p}-\frac{I}{4}\right)
+\frac{1}{4}H_{k_p}H_{k_s} \\
&= \frac{1}{4}(\frac{I}{4}+H_{k_p}H_{k_s}-2D^2_{k_p})+D^4_{k_p}.
\end{align*}
According to the proof of Theorem \ref{coercivity_thm}, we have 
\[
\left(\frac{I}{4}+H_{k_p}H_{k_s}-2D^2_{k_p}\right)\phi =
\left(-\frac{(k_s^2+k_p^2)}{2}S_0S_0+K\right)\phi
\]
where $K$ is a compact operator from $H^{1/2}(\Gamma)$ to $H^{5/2}(\Gamma)$.
Therefore, combining all the operators of $A_{11}$ yields
that 
\[
A_{11}\phi=\left(-\frac{(k_s^2+k_p^2)}{8}\overline{S}_0\overline{S}
_0+K\right)\phi, \quad\forall\phi\in H^{1/2}(\Gamma).
\]

For $A_{12}$, we apply Lemma \ref{caldlemma} for the first two terms and obtain 
\begin{align*}
\frac{1}{2}S_{k_p}T_{k_p}S_{k_p}\partial_{\tau}-\frac{1}{2}S_{k_p}\partial_{
\tau}S_{k_s}T_{k_s} &=
\frac{1}{2}\Big(H_{k_p}\big(D^2_{k_s}-\frac{I}{4}\big)-\big(D^2_{k_p}-\frac{I}{
4}\big)H_{k_p}\Big)\\
&= \frac{1}{2}\Big(H_{k_p}D^2_{k_s}-D^2_{k_p}H_{k_p}\Big),
\end{align*}
which is a compact operator from $H^{1/2}(\Gamma)$ to $H^{5/2}(\Gamma)$. For the last two terms in $A_{12}$, by Lemma
\ref{coro1}, we have 
\begin{align*}
S_{k_p}T_{k_p}S_{k_p}\partial_{\tau}
D_{k_s}+S_{k_p}\partial_{\tau}D_{k_s}S_{k_s}T_{k_s} &=
-\big(D^2_{k_p}-\frac{I}{4}\big)H_{k_p}D_{k_s} 
-H_{k_p}D_{k_s}\big(D^2_{k_s}-\frac{I}{4}\big)\\
& =  \frac{1}{2}H_{k_p}D_{k_s} - D^2_{k_p}H_{k_p}D_{k_s}-H_{k_p}D_{k_s}D^2_{k_s},
\end{align*}
which is also a compact operator from $H^{1/2}(\Gamma)$ to $H^{5/2}(\Gamma)$.
Hence we conclude that $A_{12}$ is compact from
$H^{1/2}(\Gamma)$ to $H^{5/2}(\Gamma)$.
Similar argument leads to the conclusion that $A_{21}$ is compact from
$H^{1/2}(\Gamma)$ to $H^{5/2}(\Gamma)$ and $A_{22}$ satisfies
\[
 A_{22}\phi=\left(-\frac{(k_s^2+k_p^2)}{8}\overline{S}_0\overline{S}
_0+K\right)\phi,
\]
where $K$ is compact from $H^{1/2}(\Gamma)$ to $H^{5/2}(\Gamma)$. The first
equality is proved and the second equality follows the same argument. 
\end{proof}

Now let us consider the adjoint operator of $M$ in $L^2(\Gamma)^2$:
\begin{eqnarray*}
M'=\begin{bmatrix}
T_{k_p} & -(\frac{1}{2}\partial_{\tau}+ D'_{k_p}\partial_{\tau}) \\
-(\frac{1}{2}\partial_{\tau}+ D'_{k_s}\partial_{\tau}) & -T_{k_s}
\end{bmatrix}.
\end{eqnarray*}
The adjoint operator of $N$ is 
\begin{eqnarray*}
N'=
\begin{bmatrix}
S_{k_p} & 0 \\
0 & S_{k_s}
\end{bmatrix}
\begin{bmatrix}
T_{k_p}& -(\frac{1}{2}\partial_{\tau}- D'_{k_p}\partial_{\tau}) \\
-(\frac{1}{2}\partial_{\tau}- D'_{k_s}\partial_{\tau}) & -T_{k_s}
\end{bmatrix}
\begin{bmatrix}
S_{k_p} & 0 \\
0 & S_{k_s}
\end{bmatrix}.
\end{eqnarray*}

Using a similar argument as those to prove Theorem \ref{doublethm}, we
have the following properties for the adjoint operators $M'$ and $N'$. 

\begin{theorem}
For any vector function ${\bf f}\in H^{1/2}(\Gamma)^2$, the operators $M'$ and
$N'$ satisfy
\begin{align*}
(N'M'){\bf
f}=\left(-\frac{(k_s^2+k_p^2)}{8}\overline{\mathbf{S}}_0\overline{\mathbf{S}}
_0+K_1\right){\bf f}, \\
(M'N'){\bf
f}=\left(-\frac{(k_s^2+k_p^2)}{8}\overline{\mathbf{S}}_0\overline{\mathbf{S}}
_0+K_2\right){\bf f},
\end{align*}
where $K_1, K_2$ are compact operators from $H^{1/2}(\Gamma)^2$ to
$H^{5/2}(\Gamma)^2$.
\end{theorem}

By the Fredholm alternative, the following result guarantees the
existence of a unique solution to the integral equation \eqref{mainint2}. The proof
follows the same idea as proving Theorem \ref{nullthm} and is omitted for
brevity. 

\begin{theorem}\label{nullthm2}
If neither $k_s$ or $k_p$ is the eigenvalue of the interior
Neumann problem for the Helmholtz equation in $D$, then  ${\rm Ker}(M)=
{\rm Ker}(M')=\{0\}$. 
\end{theorem}

Combining Theorems \ref{doublethm} and \ref{nullthm2} and the Fredholm
alternative, we finish the proof of Theorem \ref{mainthm2}.

\subsection{Proof of Theorem \ref{mainthm3}}

Following a similar proof to Theorem \ref{doublethm}, we can show the following
results for the operators $M-{\rm i}A$ and $M'-{\rm i}A'$. 

\begin{theorem}
For any vector function ${\bf f}\in H^{1/2}(\Gamma)^2$, the operators $M-{\rm
i}A$ and $N$ satisfy 
\begin{align*}
\left(N\left(M-{\rm i}A\right)\right){\bf
f}=\left(-\frac{(k_s^2+k_p^2)}{8}\overline{\mathbf{S}}_0\overline{\mathbf{S}}
_0+K_1\right){\bf f},\\
\left(\left(M-{\rm i}A\right)N\right){\bf
f}=\left(-\frac{(k_s^2+k_p^2)}{8}\overline{\mathbf{S}}_0\overline{\mathbf{S}}
_0+K_1\right){\bf f},
\end{align*}
where $K_1, K_2$ are compact operators from $H^{1/2}(\Gamma)^2$ to
$H^{5/2}(\Gamma)^2$.
\end{theorem}

\begin{theorem}
For any vector function ${\bf f}\in H^{1/2}(\Gamma)^2$, the operators
$M'-{\rm i}A'$ and $N'$ satisfy 
\begin{align*}
\left(N'\left(M'-{\rm i}A'\right)\right){\bf
f}=\left(-\frac{(k_s^2+k_p^2)}{8}\overline{\mathbf{S}}_0\overline{\mathbf{S}}
_0+K_1\right){\bf f},\\
\left(\left(M'-{\rm i}A'\right)N'\right){\bf
f}=\left(-\frac{(k_s^2+k_p^2)}{8}\overline{\mathbf{S}}_0\overline{\mathbf{S}}
_0+K_1\right){\bf f},
\end{align*}
where $K_1, K_2$ are compact operators from $H^{1/2}(\Gamma)^2$ to
$H^{5/2}(\Gamma)^2$.
\end{theorem}

The following result concerns the uniqueness. 

\begin{theorem}
Let $\Re(k_s)>0$, $\Im(k_s)\ge 0$ and $\Re(k_p)>0$, $\Im(k_p)\ge 0$. Then
${\rm Ker}(M-{\rm i}A) = \{0\}$ and ${\rm Ker}(M'-{\rm i}A') =
\{0\}$. 
\end{theorem}

\begin{proof}
We first show the uniqueness of $M-{\rm i}A$. Assume
$(\alpha(x),\beta(x))\in{\rm Ker}(M-{\rm i}A)$ and Let  
\[
\phi(x) = (\mathcal{D}_{k_p}-i\mathcal{S}_{k_p})\alpha(x), \quad \psi(x)
= (\mathcal{D}_{k_s}-i\mathcal{S}_{k_s})\beta(x), \quad x\in \mathbb
R^2\setminus\Gamma.
\]
By the uniqueness result in Theorem \ref{uniqueresult} for the exterior problem,
it follows that $\phi(x)=\psi(x)=0$ in $\mathbb R^2\setminus\overline D$. By the
jump relations \eqref{jr}, we have for $x\rightarrow \Gamma$ from the
interior of $D$ that 
\begin{align*}
\phi(x)=-\alpha(x), &\quad \psi(x)=-\beta(x), \\
\partial_\nu \phi(x)=-{\rm i}\alpha(x), &\quad \partial_\nu
\psi(x)=-{\rm i}\beta(x).
\end{align*}
According to the Green's theorem \cite{K-99}, it holds
\begin{subequations}\label{greeiden}
\begin{align}
&{\rm i}\int_{\partial D}|\alpha(x)|^2ds = \int_{\partial
D}\overline{\phi}\partial_\nu \phi(x) ds = \int_{D} (|\nabla
\phi|^2-k^2_p|\phi|^2) dx,\\
&{\rm i}\int_{\partial D}|\beta(x)|^2ds = \int_{\partial D}\overline{\psi}
\partial_\nu \psi(x)ds = \int_{\partial D}(|\nabla
\psi|^2-k^2_s|\psi|^2) dx,
\end{align}
\end{subequations}
which implies
\begin{subequations}\label{greeinequ}
\begin{align}
\int_{\partial
D}|\alpha(x)|^2ds=-2\Re(k_p)\Im(k_p)\int_{D}|\phi|^2dx\le 0, \\
\int_{\partial D}|\beta(x)|^2ds=-2\Re(k_s)\Im(k_s)\int_{D}|\psi|^2dx\le
0. 
\end{align}
\end{subequations}
Therefore, we have $\phi = \psi=0$ in $D$, which leads to the
conclusion that ${\rm Ker}(M-{\rm i}A) = \{0\}$. 

Now let us assume $(\alpha(x),\beta(x))\in{\rm Ker}(M'-{\rm i}A')$. For $x\in
\mathbb{R}^2\setminus\Gamma$, define
\begin{align*}
\phi(x) & = \mathcal{D}_{k_p}\alpha-\mathcal{S}_{k_p}\partial_\tau\beta,\\
\psi(x) & = \mathcal{D}_{k_s}\beta+\mathcal{S}_{k_s}\partial_\tau \alpha.
\end{align*}
It follows from the jump relations that the normal derivatives of $\phi$ and
$\psi$ for $x\rightarrow\Gamma$ from the interior of $D$ are given by 
\begin{align*}
\partial_\nu \phi & = T_{k_p}\alpha-\Big(\frac{I}{2}+D'_{k_p}\Big)\partial_\tau
\beta, \\
\partial_\nu \psi & = T_{k_s}\beta+\Big(\frac{I}{2}+D'_{k_s}\Big)\partial_\tau
\alpha.
\end{align*} 
The assumption $(\alpha(x),\beta(x))\in{\rm Ker}(M'-{\rm i}A')$ implies
\[
\partial_\nu \phi-{\rm i}\phi = 0,\quad \partial_\nu \psi-{\rm i}\psi = 0.
\]
Following the same arguments as \eqref{greeiden} and \eqref{greeinequ} shows
that $\phi(x)$ and $\psi(x)$ are zero for $x\in \Gamma^-$. Using the same argument in
Theorem \ref{nullthm} gives
\[
\phi(x)=0,\quad \psi(x) = 0, \quad x\in \Gamma^+,
\]
Therefore, by the jump relations, we have $\alpha=\beta=0$, which
completes the proof.
\end{proof}

Combining all the above results and the Fredholm alternative, we finish the proof of Theorem \ref{mainthm3}. 

\section{Numerical method}

Multiple scattering of small particles, including mineral particles, liquid
cloud particles, and biological microorganisms, is an important research topic
in material sciences, climatology, and biomedical engineering. Classic multiple
scattering theory, which will be mentioned below, is restricted to circular
shaped particles. In practice, particles may be arbitrarily shaped and highly
disordered.  In this section, we introduce a fast numerical method for
the elastic obstacle scattering with multi-particles that are non-circular
and randomly located in a homogeneous and isotropic elastic background
medium. Numerical methods can be found in \cite{GG-JCP13, LKB-JCP15, LKG-OE14}
for the acoustic and electromagnetic scattering problems involving
multi-particles.

\subsection{Scattering of a single disk}

Consider a rigid disk located in a homogeneous medium with Lam\'e constants
given by $\lambda$ and $\mu$ and the angular frequency given by $\omega$. The
corresponding compressional wavenumber is $k_p$ and the shear wavenumber is
$k_s$. Let the disk be centered at the origin with radius $R$. Given an incident
compressional wave  $u_p^{\rm inc}$ and shear wave $u_s^{\rm
inc}$, one can expand them in terms of Bessel functions, which is also
called the local expansion:
\begin{subequations}\label{besselexpan} 
\begin{align}
u_p^{\rm inc}(r,\theta) & = \sum_{n=-\infty}^{\infty}a_nJ_n(k_p
r)e^{{\rm i}n\theta},\\
u_s^{\rm inc}(r,\theta) & = \sum_{n=-\infty}^{\infty}b_nJ_n(k_s
r)e^{{\rm i}n\theta},
\end{align}
\end{subequations} 
where $J_n$ is the Bessel function of order $n$. By the classic Mie
theory, the exterior elastic scattered compressional and shear wave fields can
be expanded by Hankel functions, which is also called the multipole expansion:
\begin{subequations}\label{hankelexpan} 
\begin{align}
u_p^{\rm s}(r,\theta) & = \sum_{n=-\infty}^{\infty}c_n H_{n}^{(1)}(k_p
r)e^{{\rm i}n\theta}, \\
u_s^{\rm s}(r,\theta) & = \sum_{n=-\infty}^{\infty}d_n H_{n}^{(1)}(k_s
r)e^{{\rm i}n\theta}, 
\end{align}
\end{subequations}      
where $H_{n}^{(1)}$ is the Hankel function of the first kind with order $n$.
Given the expansion coefficients $\{a_n\}$ and $\{b_n\}$ of the incident wave
and the boundary conditions
\begin{align*}
\partial_\nu (u_p^{\rm inc}+u_p^{\rm s})|_{r=R} + \partial_\tau
(u_s^{\rm i}+u_s^{\rm s})|_{r=R} = 0,  \\
\partial_\tau (u_p^{\rm inc}+u_p^{\rm s})|_{r=R} - \partial_\nu
(u_s^{\rm i}+u_s^{\rm s})|_{r=R} = 0, 
\end{align*}
we can easily find the expansion coefficients $\{c_n\}$ and $\{d_n\}$ of the
scattered fields by solving a $2\times 2$ linear system for each $n$:
\begin{eqnarray*}
\begin{bmatrix}
k_pH^{(1)'}_n(k_pR) & {\rm i}nH^{(1)}_n(k_sR)\\
{\rm i}nH^{(1)}_n(k_pR) & -k_sH^{(1)'}_n(k_sR)
\end{bmatrix}
\begin{bmatrix}
c_n \\
d_n
\end{bmatrix}
= 
-\begin{bmatrix}
a_n k_pJ'_n(k_pR)+{\rm i}nb_n J_n(k_sR)\\
{\rm i}n a_n J_n(k_pR)-b_n k_sJ'_n(k_sR)
\end{bmatrix}.
\end{eqnarray*}
Explicitly we have 
\[
 \begin{bmatrix}
c_n \\
d_n
\end{bmatrix}
= \mathscr S_n 
\begin{bmatrix}
a_n \\
b_n
\end{bmatrix},
\]
where
\[
 \mathscr S_n=-\begin{bmatrix}
k_pH^{(1)'}_n(k_pR) & {\rm i}nH^{(1)}_n(k_sR)\\
{\rm i}nH^{(1)}_n(k_pR) & -k_sH^{(1)'}_n(k_sR)
\end{bmatrix}^{-1}\begin{bmatrix}
k_pJ'_n(k_pR) & {\rm i}n J_n(k_sR)\\
{\rm i}n J_n(k_pR) & - k_sJ'_n(k_sR)
\end{bmatrix}.
\]

\begin{definition}
The mapping between the incoming coefficients $\{a_n\}$ and $\{b_n\}$
and outgoing coefficients $\{c_n\}$ and $\{d_n\}$ is referred to as the
scattering matrix for the disk and denoted by $\mathscr{S}$, i.e.,
\begin{eqnarray*}
\begin{bmatrix}
\{c_n\}\\
\{d_n\}
\end{bmatrix}
=\mathscr{S}
\begin{bmatrix}
\{a_n\}\\
\{b_n\}
\end{bmatrix}.
\end{eqnarray*}
\end{definition}

\subsection{Scattering of multiple disks}

Now let's consider $M$ ($M>1$) rigid disks with the same radius $R$. A global
expansion for the exterior field which is done for a single disk does not hold
anymore. However, if we assume that the disks are well separated, i.e., there
exists a positive distance between any two disks. In such a situation, the Mie
series expansion still holds in the vicinity of each disk. For the $m$-th disk,
the field around it can be expanded in terms of the Hankel functions
\eqref{hankelexpan} with expansion coefficients 
$\{c^m_n\}$ and $\{d^m_n\}$.  The incoming field has two components: the first
one is the external incident field, as the case for a single disk, and the
second one is the scattered field of the other disks. Therefore, in order to
find $\{c^m_n\}$ and $\{d^m_n\}$, we need to solve the linear system  
\begin{eqnarray}\label{mainlinear}
\begin{bmatrix}
\mathscr{S}^{-1} & \mathscr{T}^{12} & \cdots & \mathscr{T}^{1M} \\
\mathscr{T}^{21} & \mathscr{S}^{-1} & \cdots & \mathscr{T}^{2M} \\
\vdots  & \vdots & \ddots & \vdots \\
\mathscr{T}^{M1}  & \mathscr{T}^{M2} & \cdots & \mathscr{S}^{-1} \\
\end{bmatrix}
\begin{bmatrix}
\begin{bmatrix}
\{c^1_n\} \\
\{d^1_n\}
\end{bmatrix}\\[10pt]
\begin{bmatrix}
\{c^2_n\} \\
\{d^2_n\}
\end{bmatrix}\\
\vdots \\
\begin{bmatrix}
\{c^M_n\} \\
\{d^M_n\}
\end{bmatrix}
\end{bmatrix}
=
\begin{bmatrix}
\begin{bmatrix}
\{a^1_n\} \\
\{b^1_n\}
\end{bmatrix}\\[10pt]
\begin{bmatrix}
\{a^2_n\} \\
\{b^2_n\}
\end{bmatrix}\\
\vdots \\
\{a^M_n\} \\
\{b^M_n\}
\end{bmatrix},
\end{eqnarray}   
where the matrix $\mathscr T^{m l}, m=1, \dots, M, l=1, \dots, M$, which maps
the outgoing coefficients $\{c^l_n\}$ and $\{d^l_n\}$ of the $l$-th disk to the
incoming coefficients of the $m$-th disk, is constructed based on the Graf
addition theorem \cite{R-90}.  

\begin{lemma}\label{Graflemma}
Let disk $l$ be centered at $x_l$ and disk $m$ be centered at $x_m$.
Then the multipole expansion 
\[
\sum_{n=-\infty}^{\infty} \beta_n^l H^{(1)}_n(k r_l)e^{{\rm i} n\theta_l}
\]
from disk $l$ induces a field on disk $m$ of the form
\[
u = \sum_{n'=-\infty}^{\infty} \alpha^{ml}_{n'} J_{n'}(k r_m)e^{{\rm i}
n'\theta_m},
\]
where
\[
\alpha^{ml}_{n'} = \sum_{n=-\infty}^{\infty} 
e^{-{\rm i}n(\theta_{ml}-\pi)} \beta^l_{n'-n}
H^{(1)}_n(k |x_l -x_m |).
\]
Here $(r_l,\theta_l)$ and $(r_m,\theta_m)$ denote the polar coordinates
of a target point with respect to disk centers 
$x_l$ and $x_m$, respectively, and $\theta_{ml}$ denotes the angle between 
$(x_l - x_m)$ and the $x$-axis.
\end{lemma}

From Lemma \ref{Graflemma}, we see that the translation matrix $\mathscr T^{m l}$ has the form 
\begin{eqnarray*}
\mathscr T^{m l}=
\begin{bmatrix}
\big\{H^{(1)}_{i-j}(k_p |x_l -x_m |)e^{-{\rm
i}(i-j)(\theta_{ml}-\pi)}\big\}_{i,j\in \mathbb{Z}} & 0 \\
0& \big\{ H^{(1)}_{i-j}(k_s |x_l -x_m |)e^{-{\rm
i}(i-j)(\theta_{ml}-\pi)}\big\}_{i,j\in \mathbb{Z}}
\end{bmatrix}.
\end{eqnarray*}
Since the scattering matrix $\mathscr{S}$ is ill-conditioned, it introduces
large
numerical errors if inverting $\mathscr{S}$ directly. A better way to
solve the linear system \eqref{mainlinear} is to first introduce a block
diagonal matrix, whose diagonal blocks are the scattering matrix  $\mathscr{S}$,
as the preconditioner of \eqref{mainlinear}. Hence, instead of solving
\eqref{mainlinear}, we solve the preconditioned linear system
\begin{eqnarray}\label{precondmainlinear}
\begin{bmatrix}
I & \mathscr{S}\mathscr{T}^{12} & \cdots & \mathscr{S}\mathscr{T}^{1M} \\
\mathscr{S}\mathscr{T}^{21} & I & \cdots & \mathscr{S}\mathscr{T}^{2M} \\
\vdots  & \vdots & \ddots & \vdots \\
\mathscr{S}\mathscr{T}^{M1}  & \mathscr{S}\mathscr{T}^{M2} & \cdots & I \\
\end{bmatrix}
\begin{bmatrix}
\begin{bmatrix}
\{c^1_n\} \\
\{d^1_n\}
\end{bmatrix}\\[10pt]
\begin{bmatrix}
\{c^2_n\} \\
\{d^2_n\}
\end{bmatrix}\\
\vdots \\
\begin{bmatrix}
\{c^M_n\} \\
\{d^M_n\}
\end{bmatrix}
\end{bmatrix}
=
\begin{bmatrix}
\mathscr{S}\begin{bmatrix}
\{a^1_n\} \\
\{b^1_n\}
\end{bmatrix} \\[10pt]
\mathscr{S}\begin{bmatrix}
\{a^2_n\} \\
\{b^2_n\}
\end{bmatrix} \\
\vdots \\
\mathscr{S}\begin{bmatrix}
\{a^M_n\} \\
\{b^M_n\}
\end{bmatrix} \\
\end{bmatrix}.
\end{eqnarray}  
Due to the existence of positive distance between any two disks, the translation
operators $\mathscr{T}^{ml}, m=1, \dots, M, l=1, \dots, M$ are compact
and the scattering matrix $\mathscr{S}$ is bounded, which implies  the
system \eqref{precondmainlinear} is much better conditioned than the
original system \eqref{mainlinear}. Therefore, one can apply an iterative
solver, such as GMRES, to the system \eqref{precondmainlinear} and expect a
fast convergence rate. For numerical purpose, all the infinite series $\{a_n\}$,
$\{b_n\}$, $\{c_n\}$, and $\{d_n\}$ need to be truncated to a finite number of terms with
$N>0$. Since the linear matrix in \eqref{precondmainlinear} is dense, direct
matrix-vector product in each iteration takes a computational complexity on the
order of $O(M^2)$ if the truncation number $N$ is relatively small. In this
case, the FMM can be utilized to reduce the complexity to $O(M)$
in each iteration and greatly accelerate the computation \cite{GR-JCP87} .

\subsection{Scattering of arbitrarily shaped
multiple obstacles}\label{arbparticle}

The theory described above for the elastic scattering of multiple disks is based
on the classic acoustic multiple scattering theory, which is efficient to find
the scattering field for a large number of disks. It is not easy, however, to
extend to the scattering of non-circular shaped particles. Here, we propose a
fast algorithm for the scattering of a large number of arbitrarily shaped
multi-particles, which are assumed to be well separated in the sense that each
particle is included in a disk, and all the disks do not overlap. Given  such an
assumption, we construct the scattering matrix $\mathscr{S}$ for each
particle based on the disk that includes the particle and extend the multiple
scattering theory to non-circular particles. 

More specifically, given $M$ randomly located particles $D_j$,
$j=1, \dots, M$, each particle is included by a non-overlapping disk $B_j$,
$j=1,\dots, M$. We sample the incoming field on the disk $B_j$ rather than
$D_j$ in the form of \eqref{besselexpan}. In particular, for each
$n\in\{-N,\dots,N\}$, let $\alpha_n$ and $\beta_n$ denote the solution to the 
integral equation \eqref{mainint} with right-hand side given by
\eqref{besselexpan}, where we choose $a_n = 1$ or $b_n=1$ with $n$ sequentially
being $-N,-N+1,\cdots,N-1,N$ and $a_m = 0$ or $b_m=0$ for any $-N\le m\le N$ and
$m\ne n$. Then we precompute the multipole expansion
\eqref{hankelexpan} 
from these source distributions, where
\begin{align*}
c^n_l & = \int_{\Gamma_j} 
J_l(k_p|y|) e^{- {\rm i} l \theta_j(y)} \, 
\alpha_n(y)ds(y), \\
d^n_l & = \int_{\Gamma_j} 
J_l(k_s|y|) e^{- {\rm i} l \theta_j(y)} \, 
\beta_n(y)ds(y),
\end{align*}
$l = -N, \dots, N$. Here, $y$ is the location of a point on 
$\Gamma_j$ with respect to the center of the disk $B_j$ and
$\theta_j(y)$ is the polar angle subtended with respect to
the center of disk $B_j$. The formulas for $c^n_l$ and $d^n_l$ are standard 
\cite{R-90} and derived from Graf's addition theorem
\cite{OLBC-10}. Note that we only have to solve the integral equation
\eqref{mainint} by the LU factorization or any other direct solver once and
apply it to different hand sides.  Once the computation is done for each
$n\in\{-N,\dots, N\}$, we obtain the scattering matrix $\mathscr{S}_j$, which is given by the following definition.   

\begin{definition}
The mapping between the incoming coefficients $\{a_n\}$ and $\{b_n\}$
and outgoing coefficients $\{c_n\}$ and $\{d_n\}$ is referred to as the
scattering matrix for the inclusion $D_j$ and denoted by $\mathscr{S}_j$.
\end{definition}

\begin{remark}
Here we construct the scattering matrix for a given particle based on
the integral formulation \eqref{mainint} under the assumption of Theorem
\ref{nullthm}. If this assumption is not satisfied, we can use either
\eqref{mainint2} or \eqref{combsys} to find the scattering matrix.
\end{remark}

When the scattering matrix $\mathscr{S}_j$ for each particle is available, we
plug them into the linear system \eqref{precondmainlinear} to find the
elastic scattered field. The advantage of using scattering matrix, instead of
points that discretize each particle directly, is that the number of unknowns
represented by multipole expansion coefficients is usually much less than the
one represented by points, especially for particles with complicated geometries.
Moreover, by using the scattering matrix, we obtain a much better conditioned
system and GMRES can find the solution rapidly. In addition, if all the
particles are identical up to a rotation, we only have to compute the scattering
matrix $\mathscr{S}$ for one particle and apply it to all the other particles. 

\section{Numerical experiments}

In this section, we test our algorithm by evaluating the elastic scattered field
for a large number of identical particles embedded in a homogeneous and
isotropic background. Particles tested in all the examples, up to a rotation and
shift,  are parametrized by 
\begin{align*}
x(\theta)  = (a+b\cos(c\theta))\cos\theta,\quad 
y(\theta)  = (a+b\cos(c\theta))\sin\theta,
\end{align*}
where $\theta\in[0,2\pi)$ and the parameters $a, b, c$ will be specified in each
example. For simplicity, we fix the Lam\'e constants to be $\lambda= 3.88$ and
$\mu = 2.56$ in all examples and change the angular frequency $\omega$ only. In
order to discretize the singular
integral accurately, we use the Nystr\"{o}m discretization for the system of
equations \eqref{mainint} based on the high order hybrid Gauss-trapezoidal rule
of Alpert \cite{A-SISC99}. 

The following notations are given in the table to illustrate the results:
\begin{itemize}
\item $\omega$: the angular frequency,
\item $N_{pts}$: the number of points to discretize a single particle,
\item $N_{particle}$: the total number of particles,
\item $N_{term}$:  the highest order that is used in the local and multipole
expansion of a single particle, i.e., the local and multipole
expansion have $2N_{term}+1$ terms.
\item $N_{tot}$: the total number of unknowns in the linear equation. If the
unknowns are given by points, then it is equal to $2N_{pts}N_{particle}$. If the
unknowns are given by the coefficients of multipole expansion, then it is equal
to $2(2N_{term}+1)N_{particle}$,
\item $N_{iter}$: the number of GMRES iterations,
\item $T_{solve}$: the time  (secs.) to solve the linear system by GMRES,
\item $E_{error}$: the relative $L^2$ error of the elastic field measured at a
few random points.
\end{itemize}

All experiments were implemented in \textsc{Fortran 90} and
carried out on a laptop with an Intel CPU and 16 GB of memory. We made
use of the simple LU-factorization for matrix inversion when constructing the
scattering matrix given in Section \ref{arbparticle}. The accuracy for GMRES
was chosen to be 1E-9. No further acceleration was explored
during the GMRES iteration except for using the FMM. 

\subsection{Example 1: scattering with analytic solution}

In this example, we consider the elastic scattering of $10$ particles, denoted
by $D_i$, $i = 1, \dots, 10$.  Two methods are used for comparison. One
method is to discretize all particles by points and apply the Nystr\"om
discretization to the integral equation directly. Since the number of unknowns
is large, we do not explicitly assemble the matrix but solve it by the GMRES
with the FMM acceleration. We call it a direct method. Another one is the
proposed method by constructing scattering matrix first and solve for the
coefficients of multipole expansion, which is called the scattering matrix
based method. To verify the accuracy of these two methods, we construct an
artificial solution by letting the field outside the particles be generated by a
point source inside the first particle. In particular, we choose the exterior
elastic field to be 
\[
{\bf u} = \nabla u_p^s + {\bf curl}u_s^s=
\begin{bmatrix}
\partial_{x_1} u_p^s\\
 \partial_{x_2} u_p^s
\end{bmatrix}
 +\begin{bmatrix}
\partial_{x_2} u_s^s\\
-\partial_{x_1} u_s^s
\end{bmatrix},
\]
where 
\begin{equation}\label{pointsource}
u_p^s = H^{(1)}_0(k_p|x-x_0|), \quad
u_s^s = H^{(1)}_0(k_s|x-x_0|),\quad x\in
\mathbb{R}^2\setminus\cup_{j=1}^{10}D_j,
\end{equation}
where $x_0\in D_1$. Due to the uniqueness property, the solution can
be recovered by enforcing a boundary condition on $\Gamma_j, j
=1, \dots, 10$, that is consistent with the given ${\bf u}$. 

 Results for various angular frequencies are shown in Figure \ref{figure_ex1}
and Tables \ref{tab11}--\ref{tab12}. Table \ref{tab11} shows the results of
the direct method. It can be seen that the number of iterations grows
rapidly when the number of discretization points increases. Since the integral
equation \eqref{mainint} is not second kind in $L^2$ space, convergence rate
based on the direct method is very
slow due to the ill-conditioning of the matrix. As shown in Table \ref{tab11},
if each particle is discretized by $200$ points, more than 1000 iterations  are
required for the GMRES to converge in all cases. Even with the FMM acceleration,
the CPU time is on the order of hundred seconds, which suggests that the direct
matrix factorization may be more efficient than the iterative method in this
case.  On the other hand, our scattering matrix based method always achieves a
quick convergence in various cases. One important feature is that the
convergence rate is almost constant for different $N_{term}$ used in the
multipole expansion if all the other factors are unchanged. If we ignore the
cost for precomputation of the scattering matrix, which is constructed by
solving the integral equation \eqref{mainint} with $200$ discretization points,
our solver is more than 1000 times faster than the direct method for the same
accuracy. 
 
Figure \ref{figure_ex1} shows the error of the computed elastic field compared
with the analytic solution when $\omega = 4\pi$ by the scattering matrix based
method. The near field is evaluated by the QBX \cite{KB-JCP13}, i.e., the field
near the boundary of each particle is evaluated by the use of local expansions
formed by the FMM. Overall the error is less than 1E-7. We also show the
comparison of convergence rate between the direct method and our method in
Figure \ref{figure_ex1}(c), \ref{figure_ex1}(d). Obviously, one can see a much
faster convergence for the scattering matrix based method.
  
\begin{figure}[t]
\centering
\includegraphics[width=0.45\linewidth]{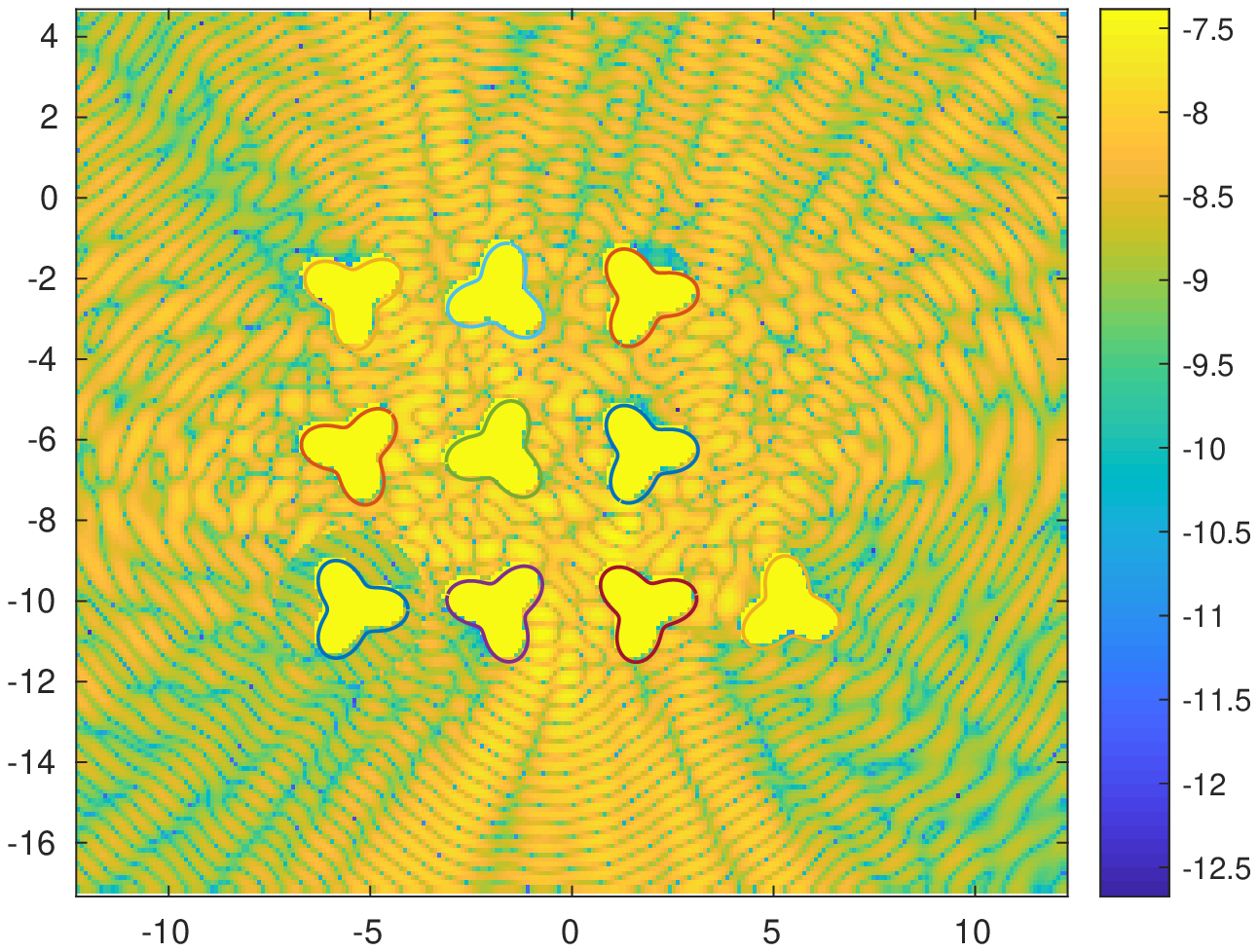}
\includegraphics[width=0.45\linewidth]{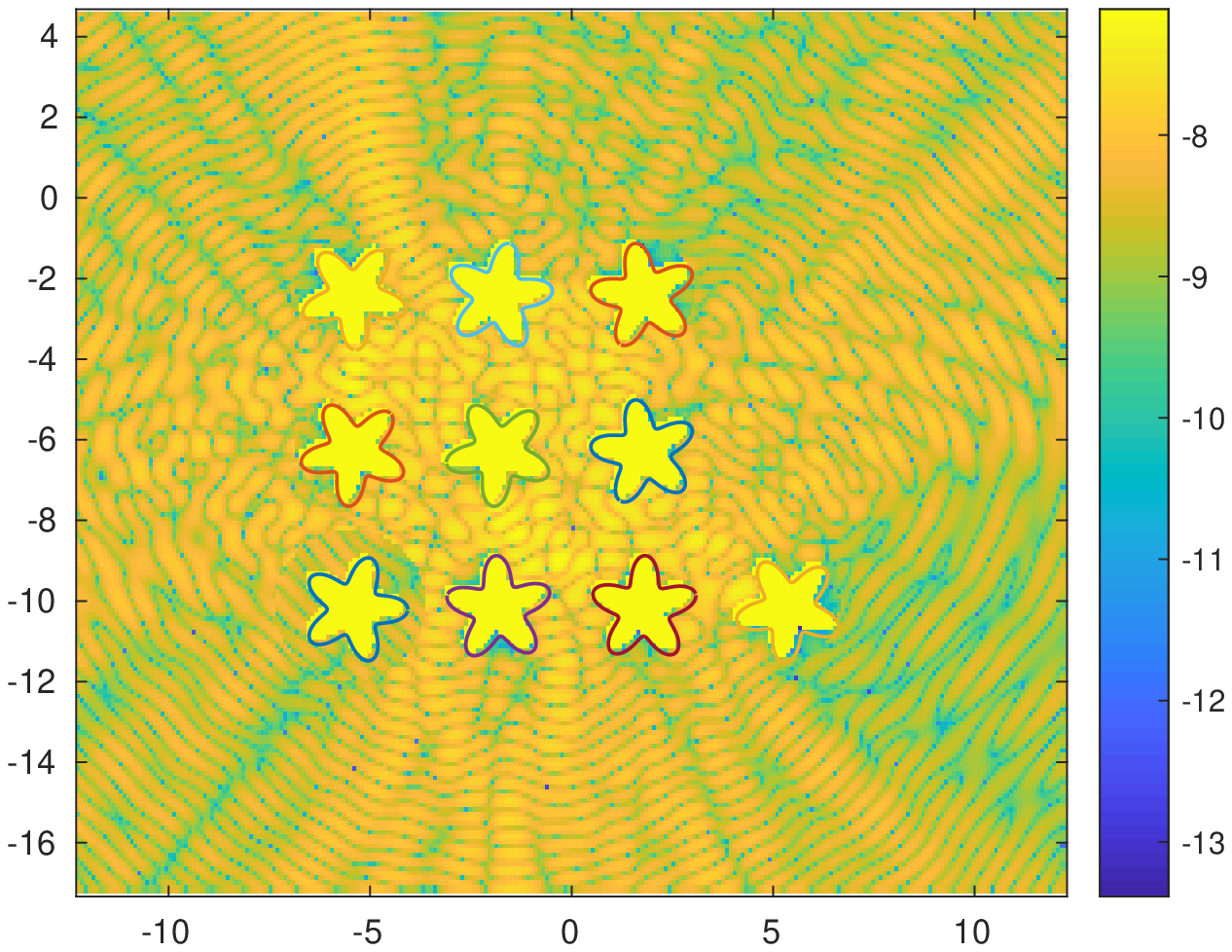}\\
(a) \hspace{7cm} (b)\\
\includegraphics[width=0.45\linewidth]{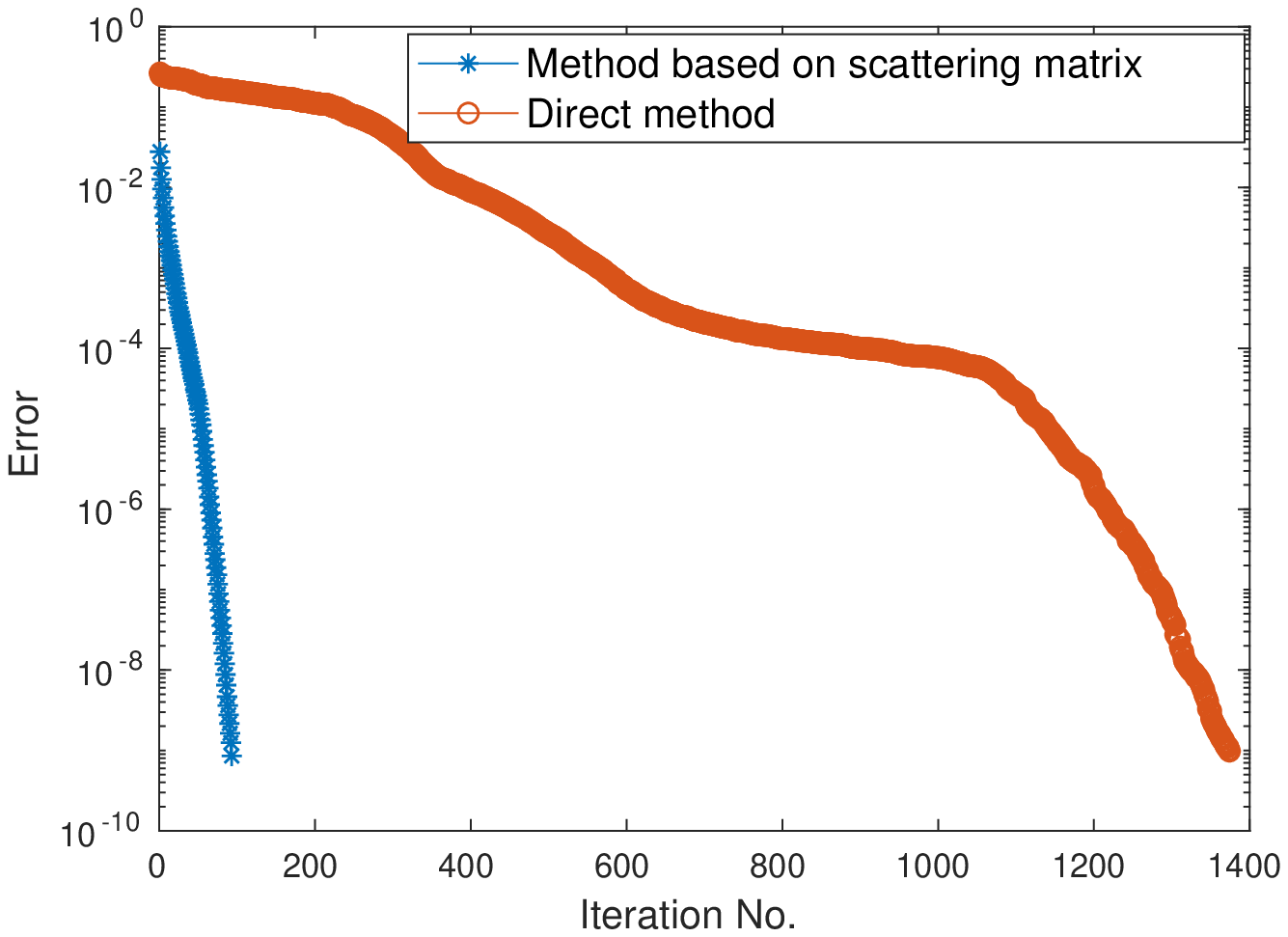}   
\includegraphics[width=0.45\linewidth]{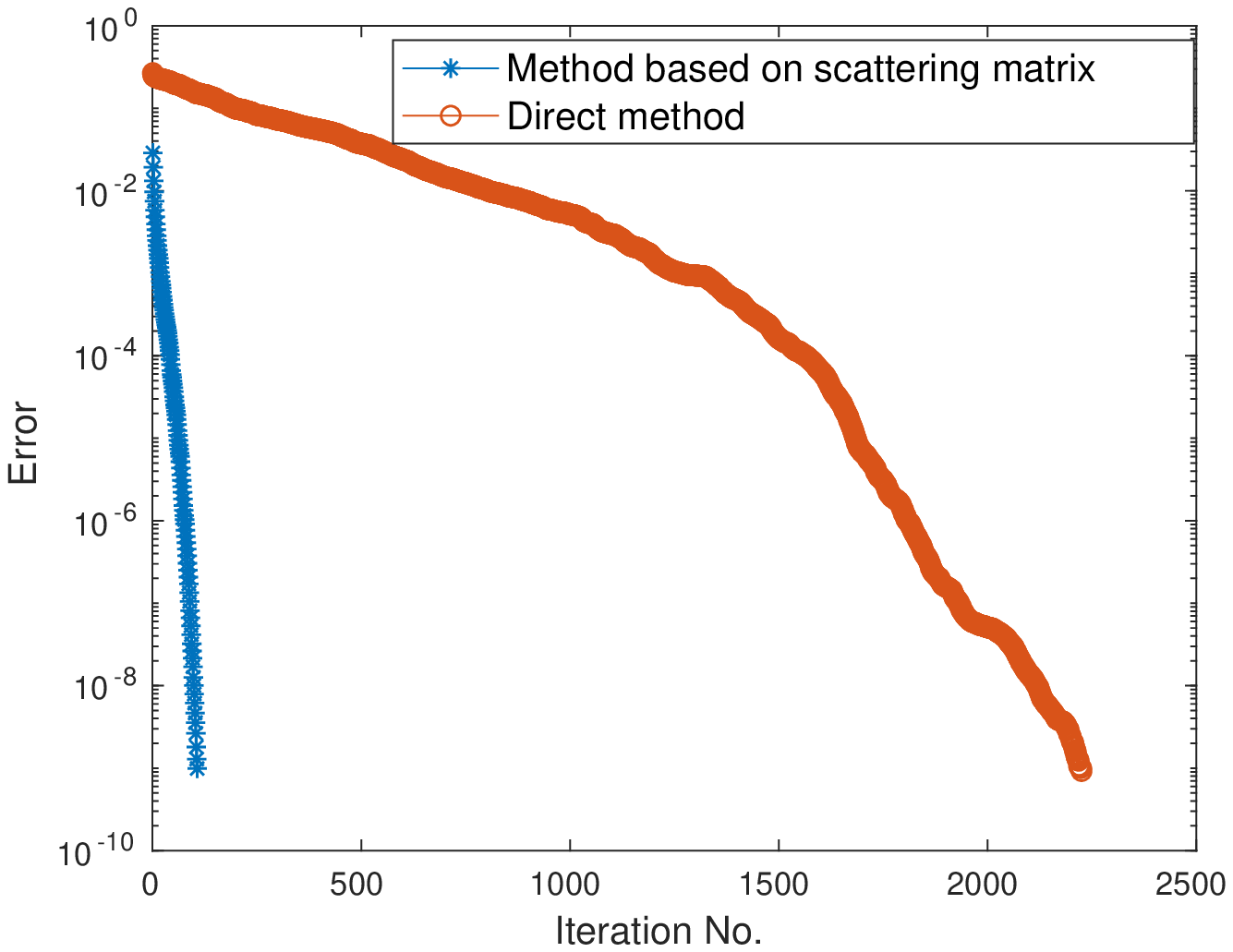}\\
(c) \hspace{7cm} (d)\\
\caption{Elastic scattering of 10 particles at $\omega=4\pi$. (a) The logarithmic error of the computed field when $a=1$, $b=1/3$,
$c=3$. (b) The logarithmic error of the computed field when $a=1$, $b=1/3$, 
$c=5$. (c) Comparison of the GMRES convergence rate when $a=1$, $b=1/3$, $c=3$.
(d) Comparison of the GMRES convergence rate when $a=1$, $b=1/3$, $c=5$. More
details are given in the text of Example 1.} \label{figure_ex1}
\end{figure}

\begin{table}
\begin{center}
\begin{tabular}{ c| c c| c c c | c c c}
\hline
&  &  & $a = 1$, &  $b = \frac{1}{3}$, & $c = 3$  & $a = 1$, &  $b
= \frac{1}{3}$, & $c = 5$   \\
\hline
$\omega$ & $N_{pts}$ & $N_{tot}$ &$N_{iter}$ & 
$T_{solve}$ & $E_{err}$  &$N_{iter}$ &  $T_{solve}$ & $E_{err}$   \\
\hline
& 50 & 1000 & 930 & 7.96E1 & 5.79E-4 & 943
& 7.84E1 & 2.12E-2  \\
$\pi$ & 100 & 2000 & 1071 & 1.79E2 &  1.36E-6&
1613 & 2.75E2 & 2.17E-7 \\
 & 200 & 4000 & 2243 & 7.23E2 &  4.47E-10 &
2814 & 1.01E3 & 2.56E-9\\
\hline
& 50 & 1000 & 930 & 7.92E1 & 5.79E-4   &
933 & 8.02E1 & 2.75E-3\\
$2\pi$ & 100 & 2000 & 1410 & 2.43E2 &  8.08E-7 &
1588 & 2.78E2 & 2.63E-5\\
 & 200 & 4000 & 2243 & 7.26E2 &  4.47E-10
& 2327 & 8.15E2 & 7.94E-9\\
\hline
& 50 & 1000 & 847 & 7.42E1 & 2.12E-3   &
993 & 8.92E1 & 1.72E-2\\
$4\pi$ & 100 & 2000 & 939 & 1.61E2 &  1.04E-6 &
1670 & 3.01E2 & 5.27E-5\\
& 200 & 4000 & 1374 & 4.29E2 &  8.11E-10
& 2225 & 8.22E2 & 2.47E-8\\
\hline
\end{tabular}
\end{center}
\caption{Example 1: Results for the elastic scattering of 10 particles based on
the direct method with FMM acceleration.} 	
\label{tab11}
\end{table}	

\begin{table}
\begin{center}
\begin{tabular}{ c| c c| c c c | c c c}
\hline
&  &  & $a = 1$, &  $b = \frac{1}{3}$, & $c = 3$  & $a =
1$, &  $b = \frac{1}{3}$, & $c = 5$   \\
\hline
$\omega$ & $N_{term}$ & $N_{tot}$
&$N_{iter}$ &  $T_{solve}$ & $E_{err}$  &$N_{iter}$ &  $T_{solve}$ & $E_{err}$  
\\
\hline
& 10 & 420 & 63 & 4.01E-2 & 8.35E-6 & 67
& 4.02E-2 & 6.59E-6  \\
$\pi$ & 20 & 820 & 62 & 8.81E-2 & 
4.71E-9 & 67 & 9.61E-2 & 6.17E-9 \\
& 40 & 1620 & 61 & 2.39E-1 &  5.75E-9 &
67 & 2.72E-1 & 6.84E-10\\
\hline
& 10 & 420 & 74 & 4.81E-2 & 3.39E-5   &
88 & 6.01E-2 & 8.41E-5\\
$2\pi$ & 20 & 820 & 73 & 1.07E-1 & 
6.49E-9 & 87 & 1.24E-1 & 5.32E-9\\
& 40 & 1620 & 73 & 2.91E-1 &  2.09E-9 &
87 & 3.61E-1 & 2.31E-9\\
\hline
& 10 & 420 & 94 & 6.39E-2 & 1.84E-2   &
104 & 7.21E-2 & 1.70E-2\\
$4\pi$ & 20 & 820 & 94 & 1.41E-1 & 
7.85E-7 & 108 & 1.64E-1 & 1.72E-7\\
& 40 & 1620 & 93 & 3.84E-1 &  3.42E-9 &
107 & 4.47E-1 & 1.12E-8\\
\hline
\end{tabular}
\end{center}
\caption{Example 1: Results for the elastic scattering of 10 particles by using
the scattering matrix based method with FMM acceleration.} 
\label{tab12}
\end{table}

\subsection{Example 2: point source incidence}

In this example, we test our algorithm on a large number of rigid particles with
point source incidence. The point source is given by the form of equation
\eqref{pointsource} with $x_0 = (5,5)$ and all the particles are
randomly located in the lower half plane. To ensure that the particles
are well separated but confined in a fixed region, we use a bin sorting
algorithm to construct the random distribution, i.e., we begin with
particles located on a regular grid and then perturb their positions randomly
several times. The details can be found in \cite{LKG-OE14}. 

We construct the scattering matrix by solving the integral equation
\eqref{mainint} on a single particle with 200 discretization points. The number
of terms in the multipole expansion is chosen to be $N_{term} = 20$. To verify
the accuracy of the computed solution, we compare it with the solution obtained
by choosing $N_{term} = 40$. Numerical results for various angular frequencies
are shown in Table \ref{tab13} and Figure \ref{figure_ex2}. From Table
\ref{tab13}, we can see that the number of iterations grows roughly linearly
with respect to the angular frequency $\omega$ for a fixed number of particles.
If $\omega$ is fixed, the number of iterations increases sublinearly with
respect to the number of particles. Another observation is that the field is
mainly affected by the size of a particle, not by the detailed geometry, since
the number of iterations is almost constant when we change the value of $c$,
which controls how many `leaves' that a particle has. The total field plotted
in Figure \ref{figure_ex2} for scattering of $1000$ particles also confirms this
observation. We have to note, however, that this conclusion may only hold when
the size of each particle is in subwavelength regime for a given incident
field.

\begin{figure}[t]
\centering 
\includegraphics[width=0.45\linewidth]{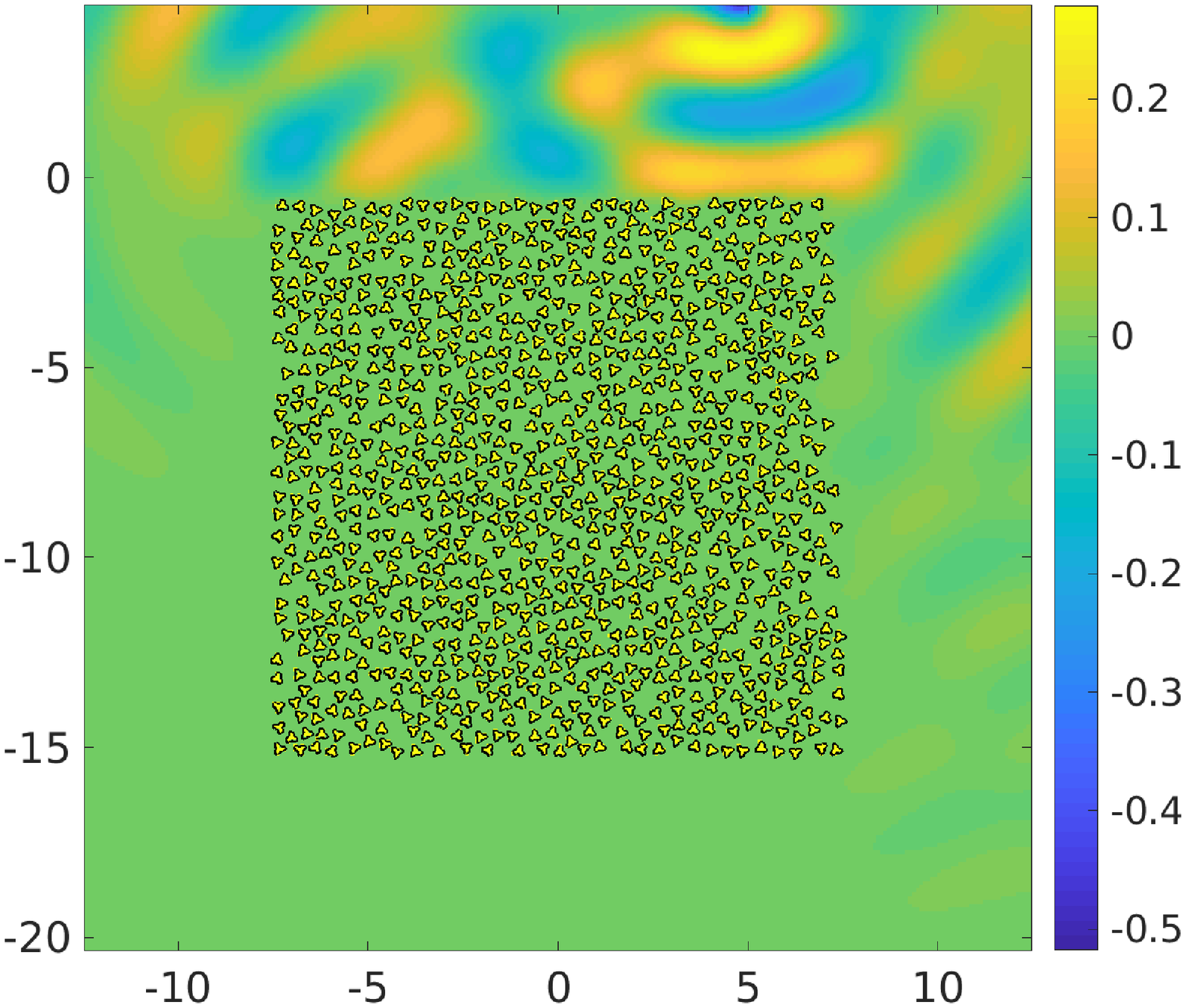}
\includegraphics[width=0.45\linewidth]{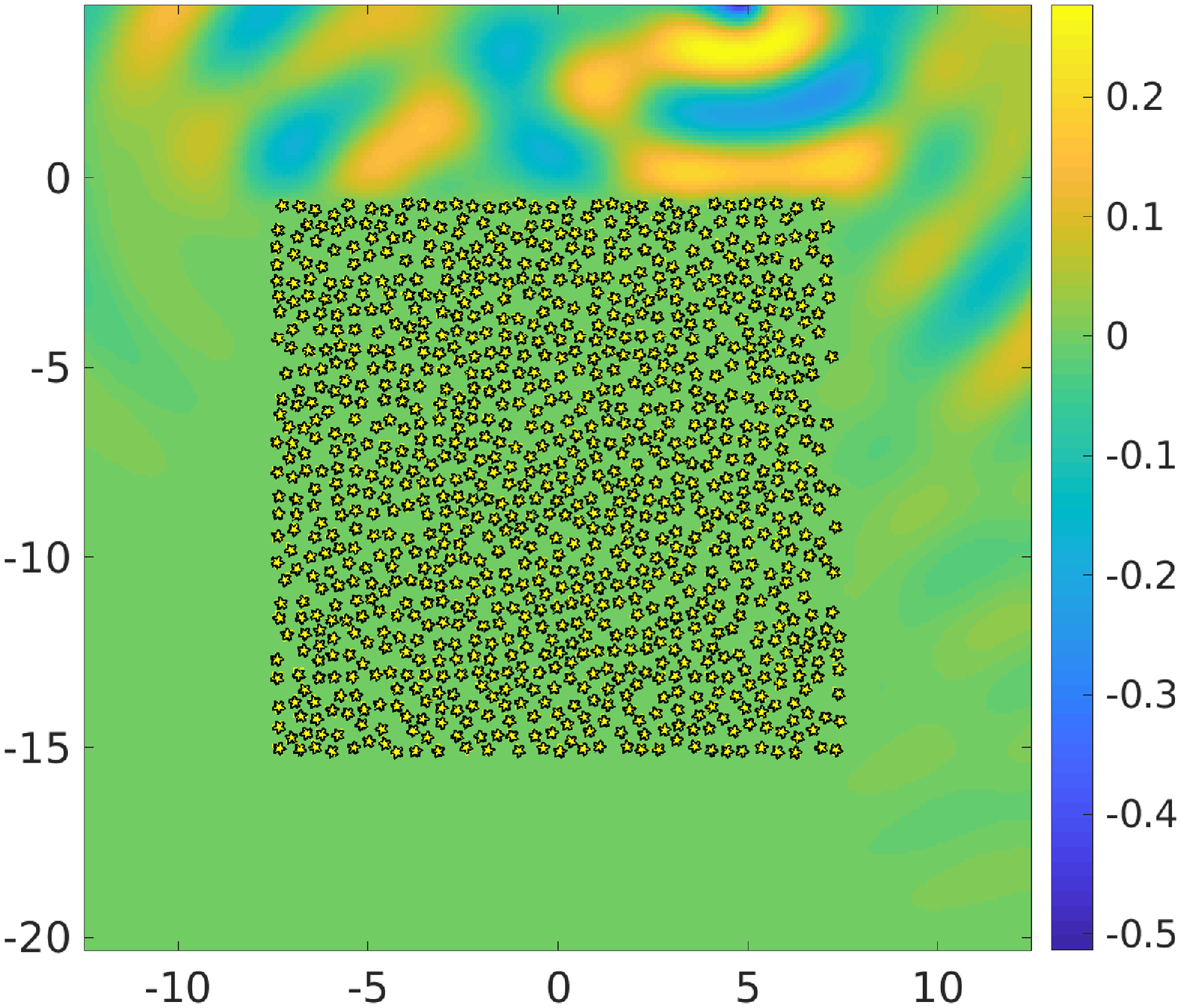}\\
(a) \hspace{7cm} (b)\\
\includegraphics[width=0.45\linewidth]{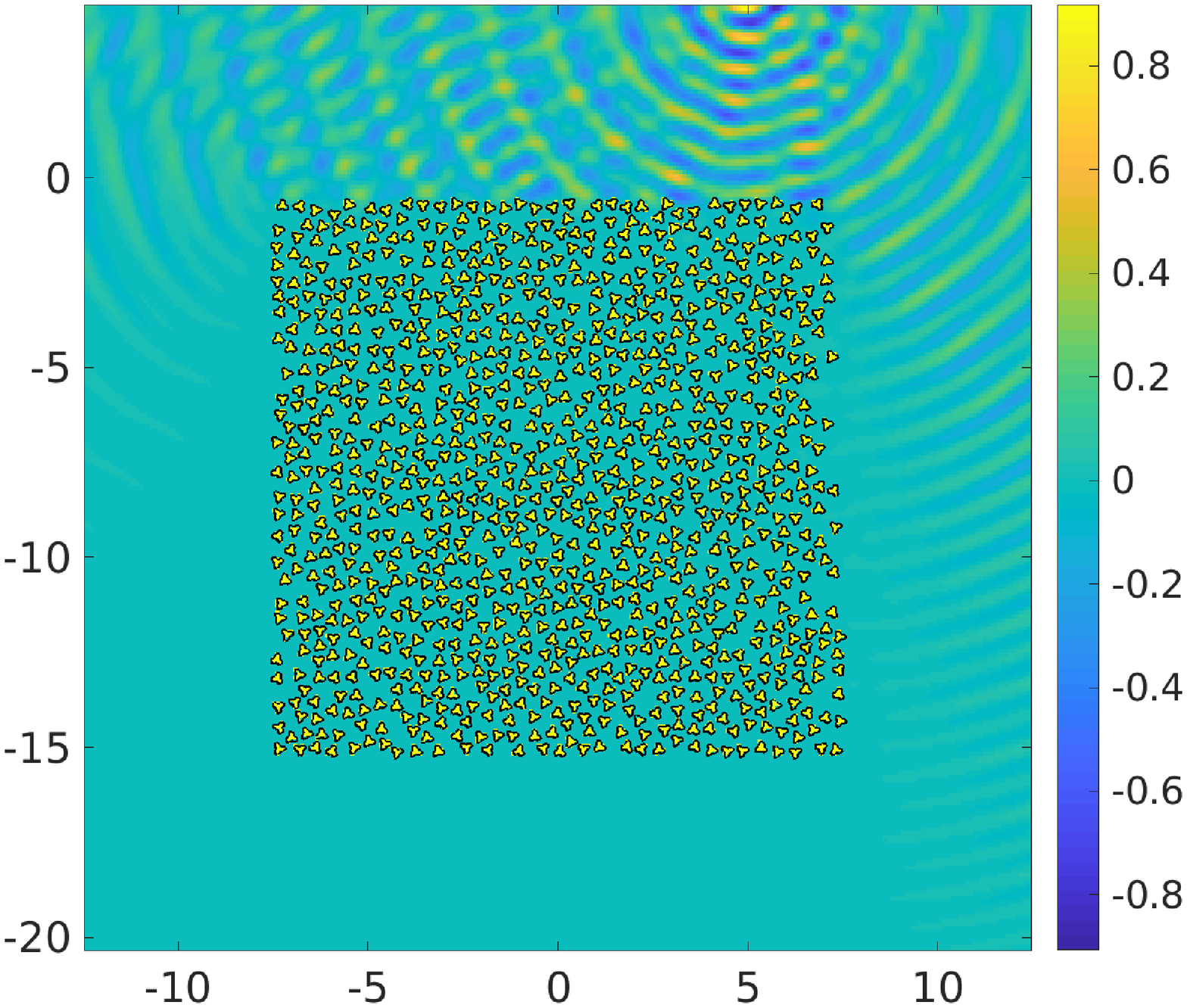}   
\includegraphics[width=0.45\linewidth]{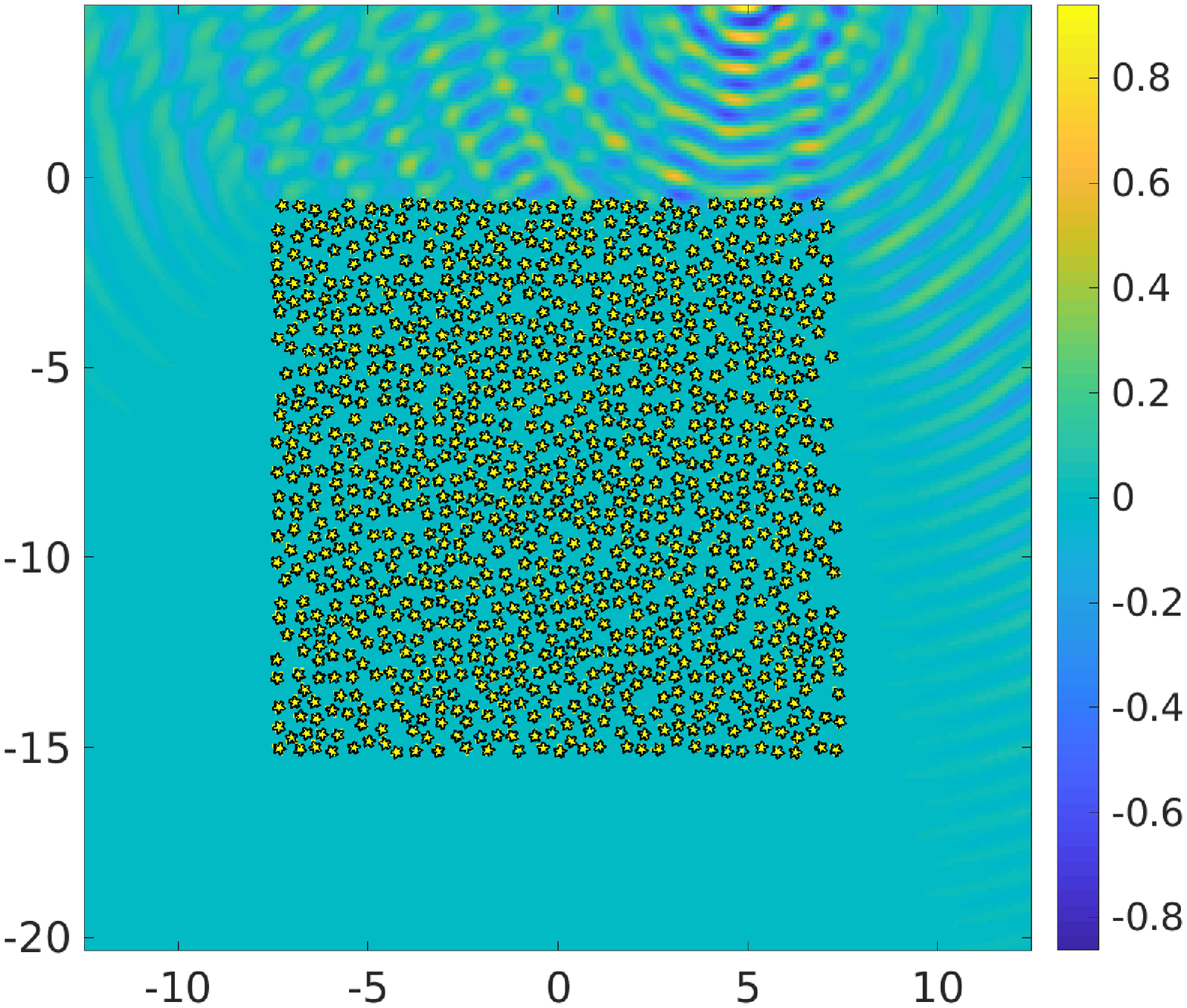}\\
(c) \hspace{7cm} (d)
\caption{The elastic scattering of 1000 particles by point source
illumination in Example 2. Here we show the real part of the first component of
the total elastic field. (a) Field for $a=\frac{1}{8}$, $b=\frac{1}{24}$,
$c=3$, $\omega=\pi$. (b) Field for $a=\frac{1}{8}$, $b=\frac{1}{24}$,
$c=5$, $\omega=\pi$. (c) Field for $a=\frac{1}{8}$, $b=\frac{1}{24}$, 
$c=3$, $\omega=4\pi$. (d) Field for $a=\frac{1}{8}$, $b=\frac{1}{24}$, 
$c=5$, $\omega=4\pi$. } \label{figure_ex2}
\end{figure}

\begin{table}
\begin{center}
\begin{tabular}{ c| c c| c c c | c c c}
\hline
&  &  & $a = \frac{1}{8}$, &  $b =\frac{1}{24}$, & $c = 3$  & $a =\frac{1}{8}
$, & $b = \frac{1}{24}$, & $c = 5$   \\
\hline
$\omega$ & $N_{particle}$ & $N_{tot}$ &$N_{iter}$ & 
$T_{solve}$ & $E_{err}$  &$N_{iter}$ &  $T_{solve}$ & $E_{err}$   \\
\hline
& 100 & 8400 & 57 & 3.51E0 & 1.38E-9 & 59 & 3.41E0 &
1.62E-10  \\
$\pi$ & 500 & 42000 & 130 & 5.28E1 &  9.39E-10 & 131 &
5.21E1 & 2.21E-9 \\
& 1000 & 82000 & 210 & 1.34E2 &  9.62E-10 & 212 & 1.36E2
& 2.81E-9\\
\hline
& 100 & 8400 & 96 & 6.09E0 & 1.64E-9   & 97 & 5.86E0 &
4.24E-9\\
$2\pi$ & 500 & 42000 & 249 & 1.05E2 &  4.10E-9 & 251 &
1.06E2 & 3.71E-9\\
& 1000 & 84000 & 347 & 2.48E2 &  1.03E-9 & 353 & 2.53E2
& 3.34E-9\\
\hline
& 100 & 8400 & 271 & 1.79E1 & 4.22E-9   & 255 & 1.65E1 &
6.31E-9\\
$4\pi$ & 500 & 42000 & 614 & 3.09E2 &  1.63E-9 & 667 &
3.43E2 & 2.54E-9\\
& 1000 & 84000 & 1197 & 1.31E3 &  7.69E-10 & 1211 &
1.34E3 & 7.42E-9\\
\hline
\end{tabular}
\end{center}
\caption{Example 2: Results for the elastic scattering of multiple
particles by using the scattering matrix based method with FMM acceleration.}	
\label{tab13}
\end{table}

\subsection{Example 3: plane incidence wave}

For the third example, we evaluate the elastic scattered field of a large number
of particles by plane wave incidence, which is given by
\begin{eqnarray*}
{\bf u}^{\rm inc} = d e^{{\rm i}k_p x\cdot d}+d^{\perp}e^{{\rm i}k_s x \cdot d},
\end{eqnarray*}
where $d$ is the propagation direction. In our test, we choose
$d = (\cos(-\frac{\pi}{3}),\sin(-\frac{\pi}{3}))$. The location of
particles are randomly distributed in a fixed region which is the same as that
in Example 2. The transformation of plane wave into the local expansion
\eqref{besselexpan} is given by the Jacobi--Anger identity\cite{OLBC-10}.
Numerical results for the plane wave incidence are given in
Figure \ref{figure_ex3} and Table \ref{tab14}. Comparing the results between
Tables \ref{tab12} and \ref{tab14}, we find that the number of iterations for
the plane wave incidence is similar to the one with the point source incidence.
In particular, the results for both the point source incidence and the plane
wave incidence show that the number of iterations for GMRES does depend on
the size of particles but is almost independent of the shape of particles. This
fact is further illustrated by Figure \ref{figure_ex3}, since the fields for two
different kind particles looks almost identical. Again, the conclusion may only
hold if we restrict in the subwavelength regime. Another observation from 
Figure \ref{figure_ex3} is that when the average distance among particles is
small, the scattered field acts as if there exists a large obstacle. How
to quantify such an equivalence will be explored in our future investigation. 

\begin{figure}[t]
\centering
\includegraphics[width=0.45\linewidth]{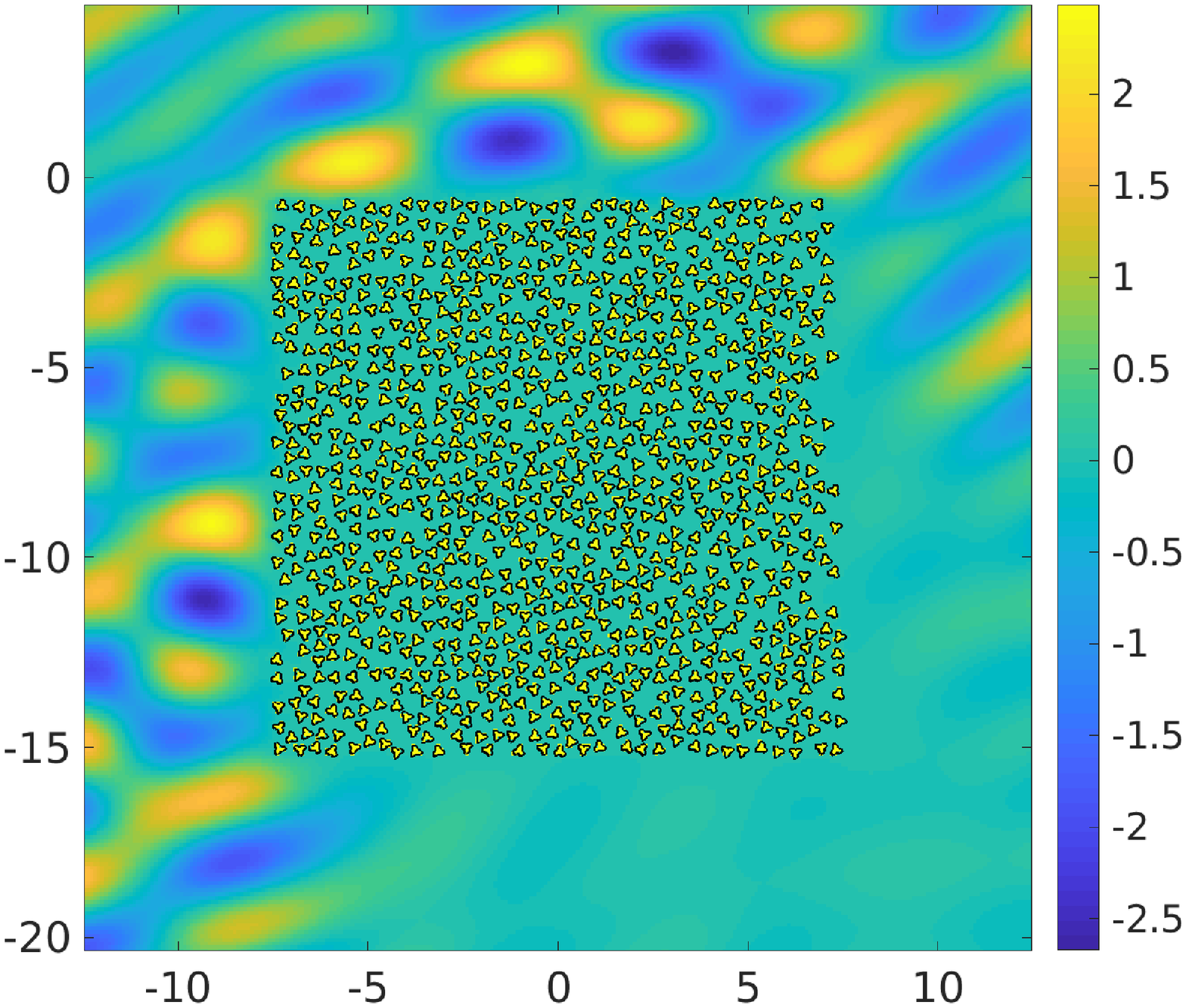}
\includegraphics[width=0.45\linewidth]{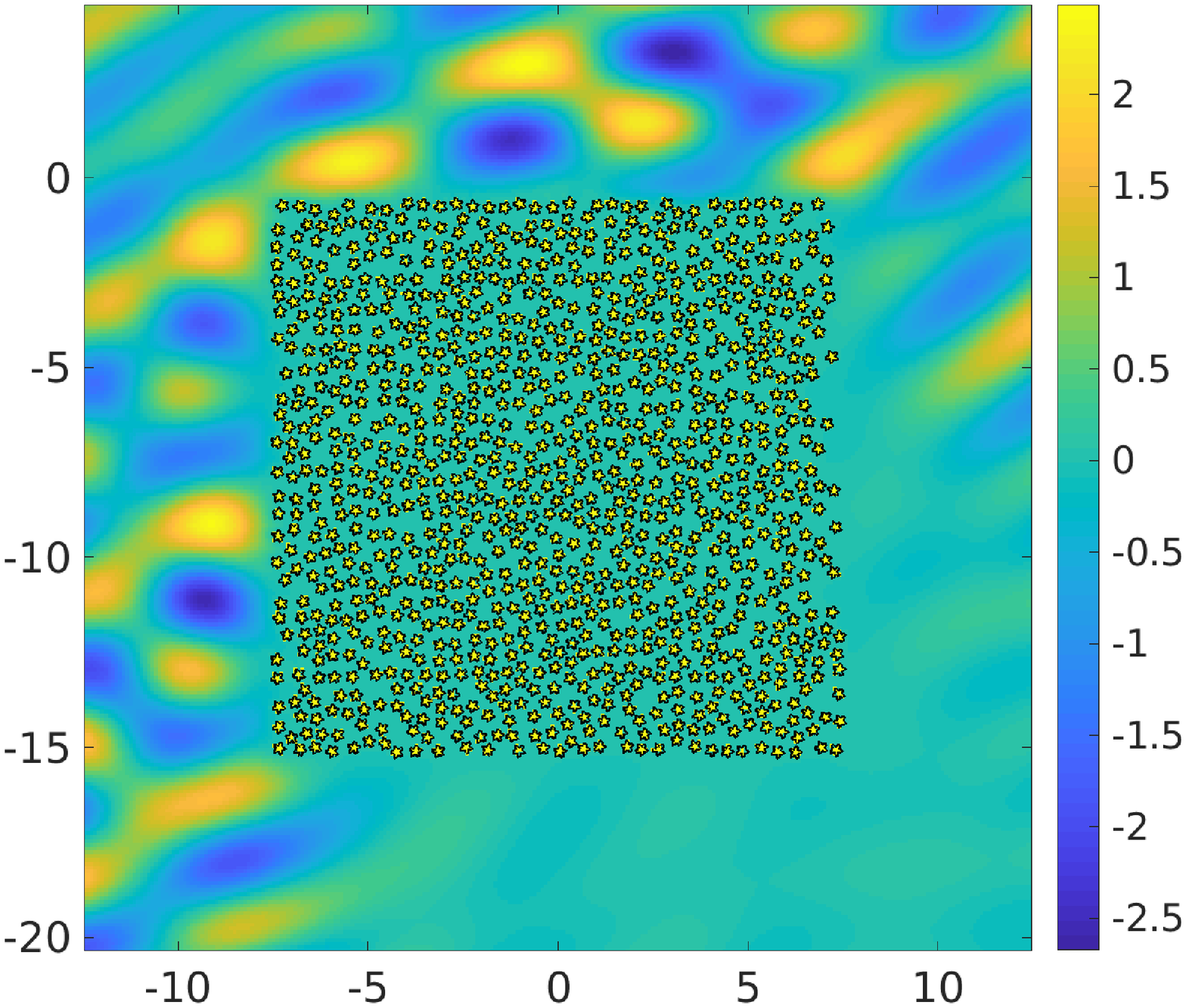}\\
(a) \hspace{7cm} (b)\\
\includegraphics[width=0.45\linewidth]{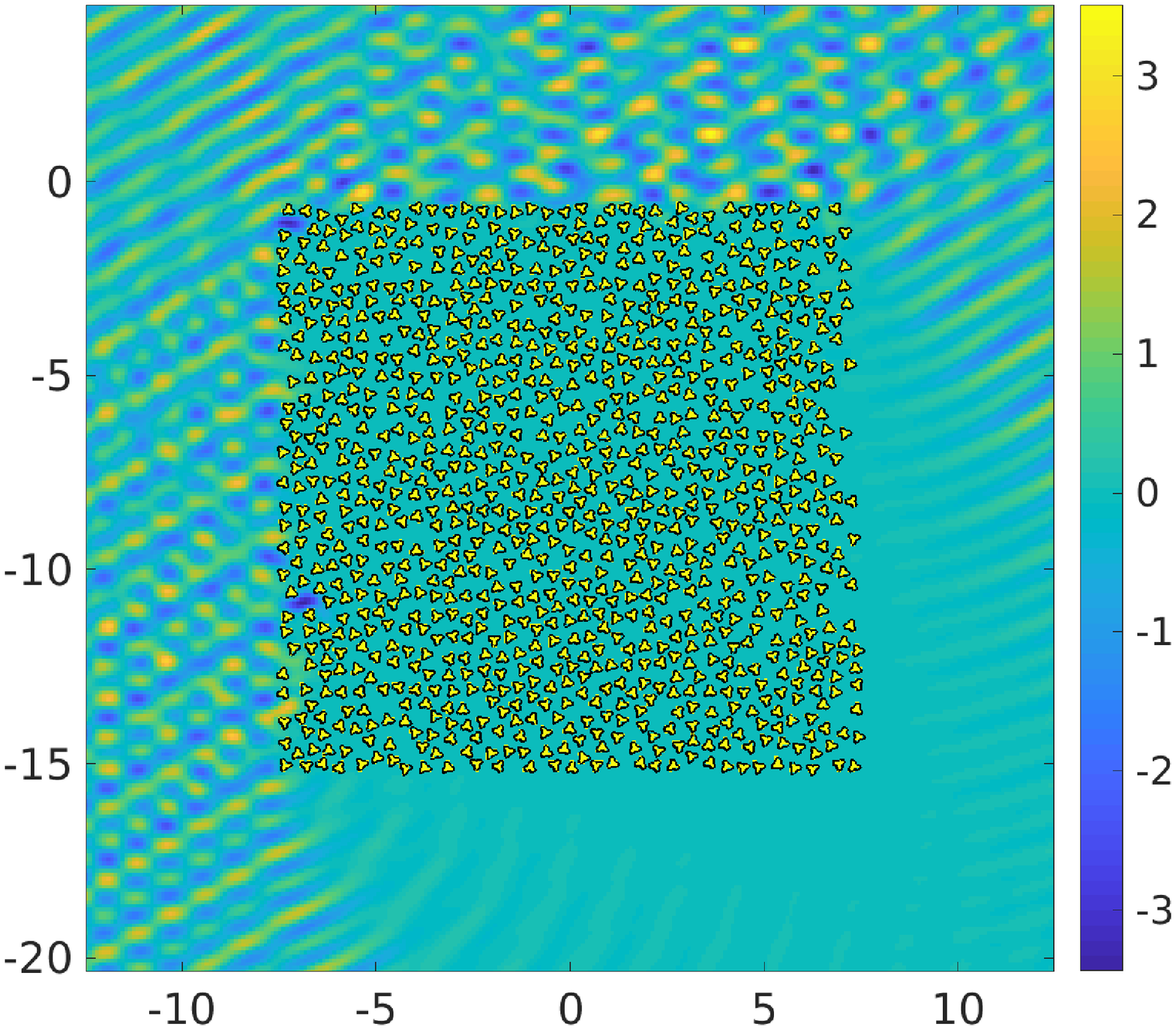}    
\includegraphics[width=0.45\linewidth]{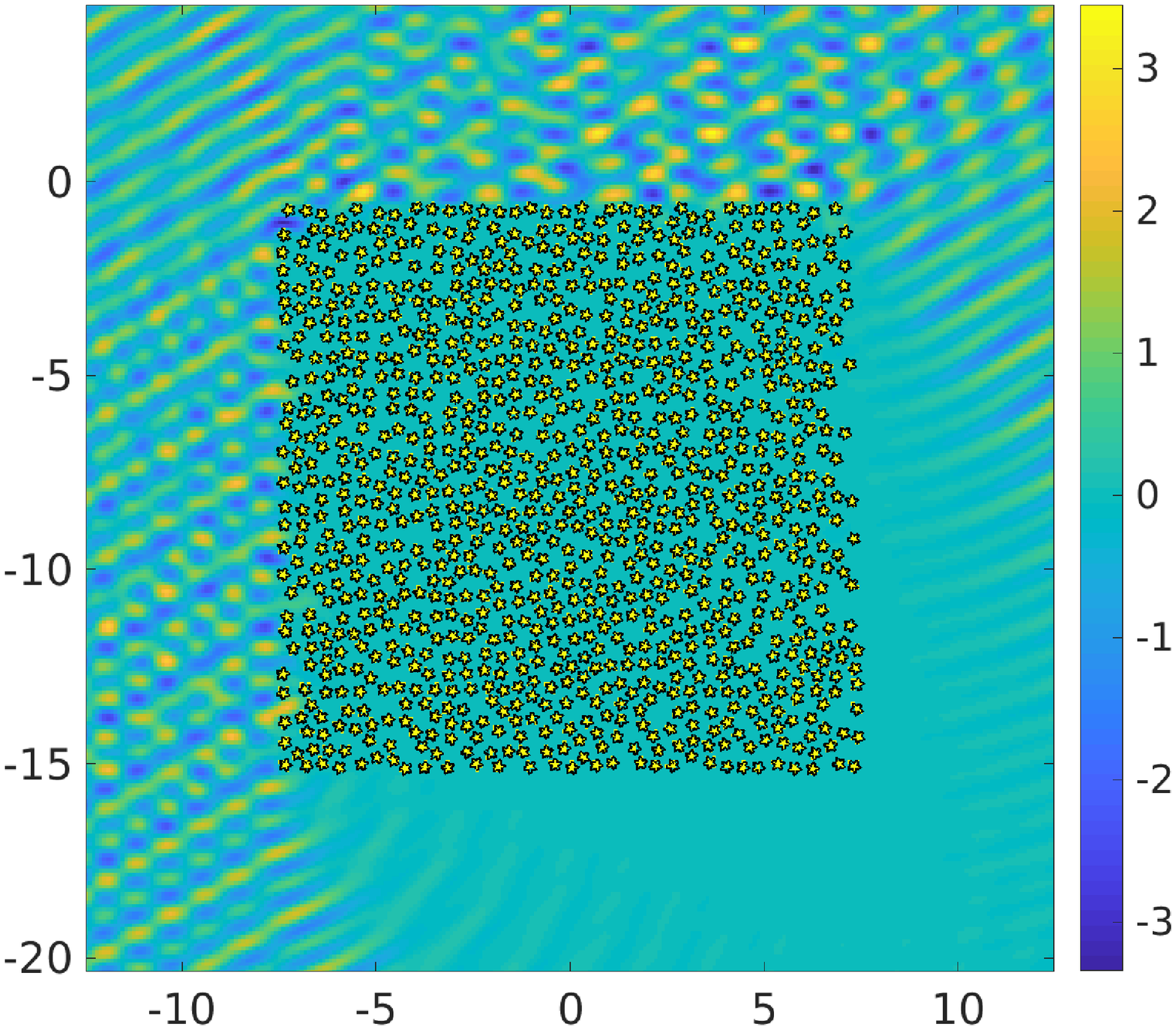}\\
(c) \hspace{7cm} (d)
\caption{The elastic scattering of 1000 particles by the plane wave incidence in
Example 3. Here we show the real part of the first component of the total
elastic field. (a) Field for $a=\frac{1}{8}$, $b=\frac{1}{24}$, $c=3$, 
$\omega=\pi$. (b) Field for $a=\frac{1}{8}$, $b=\frac{1}{24}$, $c=5$, 
$\omega=\pi$. (c) Field for $a=\frac{1}{8}$, $b=\frac{1}{24}$, $c=3$,
$\omega=4\pi$. (d) Field for $a=\frac{1}{8}$, $b=\frac{1}{24}$, $c=5$, 
$\omega=4\pi$.} \label{figure_ex3}
\end{figure}

\begin{table}
\begin{center}
\begin{tabular}{ c| c c| c c c | c c c}
\hline
&  &  &$a = \frac{1}{8}$, &  $b = \frac{1}{24}$, & $c = 3$  & $a = \frac{1}{8}
$, & $b = \frac{1}{24}$, & $c = 5$   \\
\hline
$\omega$ & $N_{particle}$ & $N_{tot}$ &$N_{iter}$ & 
$T_{solve}$ & $E_{err}$  &$N_{iter}$ &  $T_{solve}$ & $E_{err}$   \\
\hline
& 100 & 8400 & 64 & 4.49E0 & 1.75E-9 & 66 & 3.68E0 &
8.66E-10  \\
$\pi$ & 500 & 42000 & 140 & 5.62E1 &  2.21E-9 & 142 &
5.79E1 & 3.33E-9 \\
& 1000 & 82000 & 221 & 1.47E2 &  2.44E-9 & 224 & 1.55E2
& 4.97E-9\\
\hline
& 100 & 8400 & 101 & 7.36E0 & 7.76E-10   & 103 & 6.34E0
& 3.85E-9\\
$2\pi$ & 500 & 42000 & 271 & 1.17E2 &  1.32E-9 & 273 &
1.19E2 & 7.06E-9\\
& 1000 & 84000 & 384 & 2.93E2 &  2.09E-9 & 391 & 2.93E2
& 1.01E-8\\
\hline
& 100 & 8400 & 287 & 1.91E1 & 1.05E-8   & 270 & 1.89E1 &
3.48E-8\\
$4\pi$ & 500 & 42000 & 693 & 3.70E2 &  6.21E-9 & 727 &
4.10E2 & 1.01E-8\\
& 1000 & 84000 & 1459 & 1.95E3 &  1.03E-9 & 1433 &
1.78E3 & 2.62E-8\\
\hline
\end{tabular}
\end{center}
\caption{Example 3: Results for the elastic scattering of multiple
particles by using the scattering matrix based method with the FMM
acceleration.}	
\label{tab14}
\end{table}

\section{Conclusion}

In this paper, we have studied the elastic scattering problem with multiple
rigid particles by using the Helmholtz decomposition. Three different
integral formulations are presented for the coupled Helmholtz system. Their
well-posedness are studied by using appropriate regularizers. A fast
numerical method is proposed for the elastic scattering of multiple arbitrarily
shaped obstacles. The idea is to construct the scattering matrix based on the
proposed integral formulation for a single particle, and then extend the
multiple scattering theory from acoustic waves to elastic waves. In the
end, the resulted linear equation is solved by the GMRES with the FMM
acceleration. Numerical results show that our algorithm is much faster than the
one that directly discretizes particles by points. In particular, we show that
the method can achieve high order accuracy even for the scattering of up to 1000
elastic particles. The method can be extended to the three-dimensional elastic
wave scattering problem where the Helmholtz decomposition involves a scalar
potential function and a vector potential function. The progress will be
reported elsewhere in the future.


\begin{thebibliography}{10}

\bibitem{AH-SIAP76}
J. F. Ahner, G. C. Hsiao, On the two-dimensional exterior boundary-value
problems of elasticity, SIAM J. Appl. Math., 31 (1976), 677--685.

\bibitem{A-SISC99}
B. K. Alpert, Hybrid Gauss-trapezoidal quadrature rules, SIAM J. Sci. Comput.,
20 (1999), 1551--1584.


\bibitem{ABG-15}
H. Ammari, E. Bretin, J. Garnier, H. Kang, H. Lee, and A. Wahab, Mathematical
Methods in Elasticity Imaging, Princeton University Press, New Jersey, 2015.


\bibitem{BHLZ-CM14}
G. Bao, K. Huang, P. Li, and H. Zhao, A direct imaging method for inverse
scattering using the generalized Foldy--Lax formulation, Contemp. Math., 615
(2014), 49--70.


\bibitem{BXY-JCP17}
G. Bao, L. Xu, and T. Yin, An accurate boundary element method for the exterior
elastic scattering problem in two dimensions, J. Comput. Phys., 348 (2017),
343--363. 

\bibitem{BLR-JCP14}
F. Bu, J. Lin, and F. Reitich, A fast and high-order method for the
three-dimensional elastic wave scattering problem, J. Comput. Phys., 258
(2014), 856--870. 

\bibitem{CK-83}
D. Colton and R. Kress, Integral Equation Method in Scattering Theory,
Wiley-Interscience, New York, 1983.


\bibitem{GG-JCP13}
Z. Gimbutas and L. Greengard, Fast multi-particle scattering: A hybrid solver
for the Maxwell equations in microstructured materials, J. Comput. Phys.,
232 (2012), 22--32.


\bibitem{GK-WM90}
D. Givoli and J. B. Keller, Non-reflecting boundary conditions for elastic
waves, Wave Motion, 12 (1990), 261--279. 

\bibitem{GR-JCP87}
L. Greengard and V. Rokhlin, A fast algorithm for particle simulations, J.
Comput. Phys., 73 (1987), 325--348.


\bibitem{GK-JCP04}
M. J. Grote and C. Kirsch, Dirichlet-to-Neumann boundary conditions for
multiple scattering problems, J. Comput. Phys., 201 (2004), 630--650. 


\bibitem{HKS-IP13}
G. Hu, A. Kirsch, and M. Sini, Some inverse problems arising from elastic
scattering by rigid obstacles, Inverse Problems, 29 (2013), 015009

\bibitem{HL-MMS10}
K. Huang and P. Li, A two-scale multiple scattering problem, Multiscale Model.
Simul., 8 (2010), 1511--1534.

\bibitem{HLZ-JCP13}
K. Huang, P. Li, and H. Zhao, An efficient algorithm for the generalized
Foldy--Lax formulation, J. Comput. Phys., 234 (2013), 376--398.


\bibitem{JZ-CMAME12}
X. Jiang and W. Zheng, Adaptive perfectly matched layer method for multiple
scattering problems, Comput. Methods Appl. Mech. Engrg., 201 (2012), 42--52. 


\bibitem{KB-JCP13}
A. Kl\"{o}ckner, A. Barnett, L. Greengard, and M. O'Neil, Quadrature by
expansion: A new method for the evaluation of layer potentials, J. Comput.
Phys., 252 (2013), 332--349.


\bibitem{K-99}
R. Kress, Linear Integral Equations, Springer, New York, 1999.


\bibitem{LKB-JCP15}
J. Lai, M. Kobayashi, and A. Barnett, A fast and robust solver for the
scattering from a layered periodic structure containing multi-particle
inclusions, J. Comput. Phys., 298 (2015), 194--208. 


\bibitem{LKG-OE14}
J. Lai, M. Kobayashi, and L. Greengard, A fast solver for multi-particle
scattering in a layered medium, Opt. Express, 22 (2014), 20481--20499.

\bibitem{LL-86}
L. D. Landau and E. M. Lifshitz, Theory of Elasticity, Oxford: Pergamon 1986.


\bibitem{LWWZ-IP16}
P. Li, Y. Wang, Z. Wang, and Y. Zhao, Inverse obstacle scattering for elastic
waves, Inverse Problems, 32 (2016), 115018.


\bibitem{L-SIAP12}
F. Le Lou\"{e}r, On the Fr\'{e}chet derivative in elastic obstacle scattering,
SIAM J. Appl. Math., 72 (2012), 1493--1507. 

\bibitem{M-06}
P. Martin, Multiple Scattering: Interaction of Time-Harmonic Wave with
$N$ Obstacles, Encyclopedia Math. Appl. 107, Cambridge University Press,
Cambridge, 2006.


\bibitem{N-01}
J.-C. N\'{e}d\'{e}lec, Acoustic and Electromagnetic Equations: Integral
Representation for Harmonic Problems, Springer, New York, 2000.


\bibitem{OLBC-10}
F. W. J. Olver, D. W. Lozier, R. F. Boisvert, and C. W. Clark, NIST Handbook of
Mathematical Functions, Cambridge University Press, New York, 2010.


\bibitem{PV-JASA}
Y. H. Pao and V. Varatharajulu, Huygens' principle, radiation conditions, and
integral formulas for the scattering of elastic waves, J. Acoust. Soc. Amer.,
59 (1976), 1361--1371. 

\bibitem{PS-PRD73}
B. Peterson and S. Str\"{o}m, $T$ matrix for electromagnetic scattering from an
arbitrary number of scatterers and representations of E(3), Phys. Rev. D, 8
(1973), 3661--3678.


\bibitem{R-90}
V. Rokhlin, Rapid solution of integral equations of scattering theory in two
dimensions, J. Comput. Phys., 86 (1990), 414--439.


\bibitem{S-49}
A. Sommerfeld, Partial Differential Equations in Physics, Academic Press, New
York, 1949. 


\bibitem{TC-JCP07}
M. S. Tong and W. C. Chew, Nystr\"{o}m method for elastic wave scattering by
three-dimensional obstacles, J. Comput. Phys., 226 (2007), 1845--1858. 


\bibitem{TC-JCP09}
M. S. Tong and W. C. Chew, Multilevel fast multipole algorithm for elastic wave
scattering by large three-dimensional objects, J. Comput. Phys., 228 (2009),
921--932. 

\bibitem{YHX-SINUM16}
T. Yin, G. C. Hsiao, and L. Xu, Boundary integral equation methods for the
two-dimensional fluid-solid interaction problem, 55 (2017), SIAM J.
Numer. Anal., 2361--2393.


\bibitem{YLLY-CiCP18}
J. Yue, M. Li, P. Li, and X. Yuan, Numerical solution of an inverse obstacle
scattering problem for elastic waves via the Helmholtz decomposition, Commun.
Comput. Phys., to appear. 


\end{thebibliography}
\end{document}